\documentclass[12pt]{amsart}

\usepackage{comment}
\usepackage[normalem]{ulem}
\usepackage[colorlinks,linkcolor=red,anchorcolor=blue,citecolor=blue]{hyperref} 
\usepackage{amscd} 
\usepackage{amsmath}
\usepackage{mathtools}
\usepackage{amssymb} 
\usepackage{amsthm}
\usepackage{dsfont}
\usepackage{bigdelim}
\usepackage{enumerate}
\usepackage{graphicx}
\usepackage{mathrsfs}
\usepackage{multirow}
\usepackage[all]{xy} 
\usepackage{color} 
\usepackage{enumitem}
\usepackage[dvipsnames]{xcolor}
\usepackage{bbm}
\usepackage{tikz}
\usepackage{tikz-cd}

\newtheorem{thm}{Theorem}[section] 

\newtheorem{cor}[thm]{Corollary}
\newtheorem{prop}[thm]{Proposition}

\newtheorem{lem}[thm]{Lemma}
\theoremstyle{definition} 
\newtheorem{defn}[thm]{Definition}

\newtheorem{ques}[thm]{Question}

\theoremstyle{remark}
\newtheorem{rem}[thm]{Remark}
\newtheorem{setup}[thm]{Setup}

\numberwithin{equation}{section}

\newcommand{\rk}[0]{\operatorname{rk}}
\newcommand{\supp}[0]{\operatorname{Supp}}

\newcommand{\codim}[0]{\operatorname{codim}}

\newcommand{\Coker}[0]{\operatorname{Coker}}

\newcommand{\Hom}[0]{\mathscr{H}\!\textit{om}}

\newcommand{\Tor}{{\rm Tor}}

\newcommand{\reg}{{\rm{reg}}}

\newcommand{\underalign}[2]{\quad \underset{\mathclap{\strut #1}}{#2}\quad}
\newcommand{\polar}{\Omega}
\newcommand{\orb}{{\rm orb}}
\newcommand{\tor}{{\rm Tor}}
  
\newcommand{\R}{\mathbb{R}}

\newcommand{\C}{\mathbb{C}}
\newcommand{\Q}{\mathbb{Q}}

\title[Semipositivity of the orbifold second Chern class in Fujiki's class]
{Semipositivity of the orbifold second Chern class \ in Fujiki's class}

\author{Masataka IWAI}
\address{Department of Mathematics, Graduate School of Science, The University of Osaka,
1-1, Machikaneyama-cho, Toyonaka, Osaka 560-0043, Japan.}
\email{{\tt masataka@math.sci.osaka-u.ac.jp}}
\email{{\tt masataka.math@gmail.com}}

\author{Satoshi Jinnouchi}
\address{Department of Mathematics, Graduate School of Science, The University of Osaka,
1-1, Machikaneyama-cho, Toyonaka, Osaka 560-0043, Japan.}
\email{{\tt u122988d@ecs.osaka-u.ac.jp}}
\email{{\tt 20160312sti@gmail.com}}

\author{Shiyu Zhang}
\address{School of Science, Institute for Theoretical Sciences, Westlake University, Hangzhou 310030, China}
\email{{\tt zhangshiyu@westlake.edu.cn}}

\date{\today}


\begin{document}

\begin{abstract}
We study inequalities for orbifold second Chern classes of compact normal analytic varieties in Fujiki's class. We prove Miyaoka's inequality for singular varieties in Fujiki's class with nef canonical divisor, as well as the semipositivity of the orbifold second Chern class for varieties with nef anti-canonical divisor. To prove these results, we establish generic nefness theorems for tangent and cotangent sheaves and an orbifold Bogomolov--Gieseker inequality for mixed polarizations.
 \end{abstract}
 
\maketitle

\setcounter{tocdepth}{1}
\tableofcontents

\section{Introduction}

Miyaoka~\cite{Miy77} proved that a smooth complex surface \( S \) of general type satisfies the following inequality:
\begin{equation*}
    3c_2(S)-c_1(S)^2 \geq 0.
\end{equation*}
This inequality was generalized in \cite{Miy87} to higher-dimensional varieties. 
More precisely, let \( X \) be an \( n \)-dimensional normal non-uniruled projective variety that is smooth in codimension two. If the canonical divisor \( K_X \) is nef \(\mathbb{Q}\)-Cartier, then for any ample divisors \( H_1, \ldots, H_{n-2} \), we have
    \begin{equation}
    \label{equa-intro-miyaoka}
        \big(3c_2(X) - c_1(X)^2\big) \cdot H_1 \cdots H_{n-2} \geq 0.
    \end{equation}
This inequality \eqref{equa-intro-miyaoka} is often called \emph{Miyaoka's inequality}. It was later extended in \cite{IMM24} to the case where $X$ is a projective klt variety.

Inequalities involving the second Chern class, such as Miyaoka's inequality \eqref{equa-intro-miyaoka}, also play an important role in the Minimal Model Program (MMP), particularly in relation to the abundance conjecture in dimension three. In \cite{Miy88b, Miy88a}, Miyaoka used this inequality to prove the conjecture when the numerical Kodaira dimension is one. Later, Kawamata completed the three-dimensional abundance conjecture in \cite{Kaw92} by using Miyaoka's inequality for the orbifold second Chern class \( \widehat{c}_2(X) \). These techniques have also been applied to the abundance conjecture for three-dimensional K\"ahler varieties (cf.~\cite{CHP16, DO23, GP24}). 

Motivated by these works, we investigate in arbitrary dimension how far the singularity assumptions can be weakened while preserving inequalities involving the second Chern class for varieties that are not necessarily K\"ahler, with a view toward applications to the abundance conjecture.

First, we investigate Miyaoka's inequality \eqref{equa-intro-miyaoka} in the case where $K_X$ is nef. In this case, the inequality holds even when $X$ is not K\"ahler.
\begin{thm}
\label{thm-Miyaoka-ineq-nonuniruled}
Let $X$ be an $n$-dimensional compact normal complex analytic variety admitting a nef and big class. Assume one of the following conditions holds:
\begin{enumerate}[label=$(\arabic*)$]
\item $X$ has at most canonical singularities.
\item $X$ is a non-uniruled variety with at most quotient singularities in codimension $2$ and rational singularities, such that $K_X$ is $\mathbb{Q}$-Cartier.
\end{enumerate}

If $K_X$ is nef, then for any nef and big classes $\alpha_1, \ldots, \alpha_{n-2} \in H^{1,1}_{BC}(X)$, Miyaoka's inequality holds:
\[
\left(3\widehat{c}_2(X) - \widehat{c}_1(X)^2 \right)\cdot \alpha_1 \cdots \alpha_{n-2} \ge 0.
\]
\end{thm}

Even when $X$ has klt singularities, Miyaoka's inequality holds under the additional assumption that $X$ is K\"ahler.
\begin{thm}
\label{thm-Miyaoka-klt-kahler}
Let \( (X, \omega) \) be an \( n \)-dimensional compact normal K\"ahler variety. Assume that \( K_X \) is nef and that \( X \) has at most klt singularities. Then there exists a constant \( \varepsilon_0 > 0 \), depending only on \( (X, \omega) \), such that for any \( 0 < \varepsilon < \varepsilon_0 \), Miyaoka's inequality holds:
\[
\left( 3\widehat{c}_2(X) - \widehat{c}_1(X)^2 \right) \cdot \left( c_1(K_X) + \varepsilon \{ \omega \} \right)^{n-2} \ge 0.
\]
\end{thm}

Here, $\widehat{c}_2(X)$ and $\widehat{c}_1(X)^2$ denote the orbifold Chern classes introduced in \cite{Kaw92, GKPT19b, GK20}. When $X$ is smooth in codimension two, these classes coincide with the usual Chern classes. Therefore, \cite{Miy87} follows from Theorem \ref{thm-Miyaoka-ineq-nonuniruled}.
At present, orbifold Chern classes can be defined for klt varieties, more precisely for varieties with quotient singularities in codimension two. Therefore, the assumptions on the singularities in Theorems \ref{thm-Miyaoka-ineq-nonuniruled} and \ref{thm-Miyaoka-klt-kahler} cannot currently be weakened any further. 

For the case where $-K_X$ is nef, we obtain the following inequality concerning the semipositivity of the orbifold second Chern class, even when $X$ is not K\"ahler.
During the preparation of an earlier version of this paper (see also Remark~\ref{rem-earlier} below), the same inequality was independently obtained in \cite[Theorem~5.3]{MWWZ25} under the assumption that $X$ is a compact K\"ahler manifold.

\begin{thm}
\label{thm-nef-anti-canonical}
Let $X$ be an $n$-dimensional compact normal complex analytic variety that admits a nef and big class.
If $-K_X$ is nef and $X$ has at most klt singularities, then for any nef and big classes $\alpha_1, \ldots,\alpha_{n-2} \in H^{1,1}_{BC}(X)$, the following inequality holds:
\[
\widehat{c}_2(X) \cdot \alpha_1 \cdots \alpha_{n-2}\ge 0.
\]
\end{thm}

As a corollary, we obtain the Miyaoka--Yau type inequality. 
To formulate this inequality, for a nef class $\alpha$ on $X$, we define its numerical dimension by
\[
\nu(\alpha) := \max \left\{ 0 \le k \le \dim X \;\middle|\; \alpha^k \not \equiv 0 \right\}.
\]
In terms of this notion, we obtain the following Miyaoka--Yau type inequality.

\begin{cor}
\label{cor-MY-canonical}
Let $(X, \omega)$ be an $n$-dimensional compact klt K\"ahler variety with nef canonical divisor $K_X$. Set $\nu := \min\{ \nu(c_1(K_X)), n-2 \}$. Then the following Miyaoka--Yau type inequality holds:
\begin{equation*}
\left(2(n+1)\widehat{c}_2(X) - n\widehat{c}_1(X)^2\right) \cdot c_1(K_{X})^{\nu} \cdot \{\omega\}^{n-2-\nu} \ge 0.
\end{equation*}
\end{cor}

\begin{cor}
\label{cor-MY-anti-canonical}
Let $(X, \omega)$ be an $n$-dimensional compact klt K\"ahler variety with nef anti-canonical divisor $-K_X$. Set $\nu := \min\{ \nu(c_1(-K_X)), n-2 \}$. Assume one of the following conditions holds:
\begin{enumerate}[label=$(\arabic*)$]
\item $-K_X$ is big and $X$ is K-semistable.
\item $-K_X$ is not big.
\end{enumerate}
Then the following Miyaoka--Yau type inequality holds:
\begin{equation*}
\left(2(n+1)\widehat{c}_2(X) - n\widehat{c}_1(X)^2\right) \cdot c_1(-K_{X})^{\nu} \cdot \{\omega\}^{n-2-\nu} \ge 0.
\end{equation*}
\end{cor}
In particular, the inequality in \cite[Theorem 1.1]{Hisa24} follows from Corollary~\ref{cor-MY-anti-canonical}. Note that the \(K\)-semistability assumption in Corollary~\ref{cor-MY-anti-canonical}~(1) is necessary even when $-K_X$ is ample. The inequality in Corollary~\ref{cor-MY-canonical} was already established in \cite[Theorem 1.1]{ZZZ25}. The proof in  \cite{ZZZ25} relies on analytic methods involving metrics, whereas we give a more elementary proof.


We now state the strategy of the proof of Theorems~\ref{thm-Miyaoka-ineq-nonuniruled}--\ref{thm-nef-anti-canonical}. 
The key ingredients of the proof are to control the minimal slope by establishing a generic nefness theorem and to control orbifold Chern classes by using orbifold structures. For the generic nefness theorem, we apply recent advances in the theory of foliations on K\"ahler manifolds due to \cite{Ou25, CP25}. For orbifold Chern classes, we establish their additivity with respect to exact sequences by using the orbifold Grothendieck--Riemann--Roch formula, and prove the Bogomolov--Gieseker inequality for orbifold Chern classes with respect to mixed polarizations, based on the ideas of \cite{CW24, ZZZ25}.
Combining these results with the Chern number calculations in \cite{Lan04}, we obtain the desired inequalities.

\begin{rem}
\label{rem-earlier}
An earlier version of the results in this paper appeared in the second part of \cite{IJZ25}, originally posted as arXiv:2507.08522. 
The original paper was subsequently divided into two papers.
The first paper \cite{IJZ25} studies the Miyaoka--Yau inequality for big (anti-)canonical divisors, whereas the present paper is devoted to the semipositivity of Chern classes and Miyaoka's inequality for nef (anti-)canonical divisors.
\end{rem}

\subsection*{Acknowledgments}
M.\,I.\ and S.\,J.\ express their gratitude to Prof.\ Ryushi Goto for valuable discussions and for answering their questions. 
S.\,Z.\ expresses his gratitude to Prof. Xi Zhang for constant encouragement.
M.\,I.\ was supported by the Grant-in-Aid for Early Career Scientists, No.\ 22K13907. S.\,Z. was supported by NSFC No. 12401073.

The authors made limited use of ChatGPT 5.6 Plus, a large language model, for language editing and mathematical discussions. All AI-assisted text was carefully reviewed and verified by the authors. The mathematical content, arguments, and proofs were developed and verified by the authors.

\section{Preliminary results}\label{sec-pre}

\subsection{Notation and conventions}\label{subsec-notation}

In this paper, we denote by $\mathbb{N}$ the set of non-negative integers. For any real vector space $M$, we write $M^\vee$ for its dual.

All complex spaces are assumed to be \textit{second countable}. A \textit{complex analytic variety}, or simply an \textit{analytic variety}, is a reduced and irreducible complex space. We  introduce the differential operator $d^c$ such that
\(
dd^c = \frac{\sqrt{-1}}{2\pi} \partial \overline{\partial}.
\)
Throughout this paper, we use the notation $dd^c$ instead of $\frac{\sqrt{-1}}{2\pi} \partial \overline{\partial}$.

Unless otherwise stated, an analytic variety $X$ has complex dimension $n$. We follow the standard notation and conventions used in \cite{KM98} concerning the Minimal Model Program. In addition, we say that a normal analytic variety $X$ is a \emph{klt variety} if the pair $(X, 0)$ is klt; equivalently, $X$ is log terminal.

Unless stated otherwise, all sheaves considered in this paper are assumed to be coherent. Let $\mathcal{E}$ be a torsion-free (coherent) sheaf on a normal analytic variety $X$. We define the dual reflexive sheaf $\mathcal{E}^\vee$ by
\(
\mathcal{E}^\vee \coloneqq \Hom(\mathcal{E}, \mathcal{O}_X).
\)
For any positive integer $m \in \mathbb{N}$, we define the reflexive tensor power by
\(
\mathcal{E}^{[\otimes m]} \coloneqq (\mathcal{E}^{\otimes m})^{\vee\vee}.
\)
Given a morphism $f \colon Y \rightarrow X$ between analytic varieties, the reflexive pullback of $\mathcal{E}$ is defined as
\(
f^{[*]}\mathcal{E} \coloneqq (f^*\mathcal{E})^{\vee\vee}.
\)
A reflexive rank-one sheaf $\mathcal{E}$ is called a \textit{$\mathbb{Q}$-line bundle} if there exists an integer $m \in \mathbb{N}$ such that $\mathcal{E}^{[\otimes m]}$ is locally free. We denote the torsion part of a sheaf by $\tor$.

\subsection{Bott--Chern cohomology and the first Chern class}
We first recall the Bott--Chern cohomology group. For forms, psh functions, and currents on analytic varieties, we refer to \cite[Subsubsection 2.2.1]{IJZ25}.

Let \(X\) be a normal analytic variety. Denote by
\(\mathcal C_X^\infty\) (resp. \(\mathcal D'_X, \mathcal H_X\))
the sheaves of smooth functions (resp. distributions, real-valued
pluriharmonic functions) on \(X\).
According to \cite[Definition 4.6.2]{BG13}, a
\((1,1)\)-form (resp. current) with local potentials is defined as a section
of \(H^0(X,\mathcal C_X^\infty/\mathcal H_X)\)
(resp. \(H^0(X,\mathcal D'_X/\mathcal H_X)\)). By definition, any such
\((1,1)\)-form (resp. current) with local potentials can be locally written
as \(\eta=dd^c u\) for some smooth function (resp. distribution) \(u\).
If the current is positive, the local potential may be chosen to be psh.

\begin{defn}[{\cite[Definition 4.6.2]{BG13},  \cite[Definition 3.1]{HP16}}]
In the above setting, the \textit{Bott--Chern cohomology} is defined by
\[
H^{1,1}_{BC}(X) := H^1(X, \mathcal{H}_X).
\]
\end{defn}
Consider the following short exact sequence:
\[
0 \to \mathcal{H}_X \to \mathcal{C}^{\infty}_{X} \to \mathcal{C}^{\infty}_X / \mathcal{H}_X \to 0,
\]
which induces a surjective map
\(
H^0(X, \mathcal{C}^{\infty}_X / \mathcal{H}_X) \twoheadrightarrow H^{1,1}_{BC}(X).
\)
Hence, any closed $(1,1)$-form $\eta$ (resp.~closed $(1,1)$-current $T$) with local potentials defines a class $\{\eta \}$ (resp.~$\{T\}$) in $H^{1,1}_{BC}(X)$. Conversely, every element of $H^{1,1}_{BC}(X)$ can be represented in this way.

For a class $\alpha \in H^{1,1}_{BC}(X)$, positivity notions are defined as in \cite[Definition 2.2]{IJZ25}. Among these notions, we mainly use the following ones. 
Let $\omega$ be a Hermitian form. A class $\alpha$ is called \emph{nef} if it can be represented by a $(1,1)$-form $\eta$ with local potentials such that, for every $\varepsilon > 0$, there exists a smooth function $f_{\varepsilon}$ satisfying
\(
\eta + dd^c f_{\varepsilon} \geq -\varepsilon \omega.
\)
The class $\alpha$ is called \emph{big} if it contains a K\"ahler current $T$.

Let $H^p(X, \R)$ denote the singular cohomology group, and $H_p(X, \R)$ denote the singular homology group.
Let $\mathbb{R}_X$ denote the constant sheaf with values in $\mathbb{R}$. By \cite[Section 1]{Wu21}, $H^p(X, \R)$ is isomorphic to the sheaf cohomology group $H^p(X, \mathbb{R}_X)$. Hence, in what follows, we identify these two groups.

According to \cite[Subsection 3.1]{GK20}, we have the following short exact sequence:
\begin{equation*}
0 \to \mathbb{R}_X \xrightarrow{\sqrt{-1}} \mathscr{O}_X \xrightarrow{\mathrm{Re}} \mathcal{H}_{X} \to 0.
\end{equation*}
This induces the following connecting homomorphism:
\begin{equation}
\label{eq-connect-h2}
\delta^1 \colon H^{1,1}_{BC}(X)=H^1(X, \mathcal{H}_{X}) \to H^2(X, \mathbb{R}_X)=H^2(X, \mathbb{R})
\end{equation}

We recall the definition of the (homological) first Chern class of a torsion-free coherent sheaf.
\begin{defn}\cite[Definition 3 and Lemma 4]{Wu22}
\label{defn-first-chern}
Let $X$ be an $n$-dimensional compact analytic variety that is smooth in codimension one. 
For any torsion-free coherent sheaf $\mathcal{E}$ on $X$, we take a resolution of singularities $\pi : \widetilde{X} \to X$ such that $\pi^*\mathcal{E}/\tor$ is locally free, and define
\[
c_1(\mathcal{E})
:= \pi_{*}
\left( 
c_1(\pi^*\mathcal{E}/\tor) \cap [\widetilde{X}]
\right)
 \in H_{2n-2}(X,\mathbb{R})
\]
Note that $c_1(\pi^*\mathcal{E}/\tor) \cap [\widetilde{X}] \in H_{2n-2}(\widetilde{X},\R)$, and $\pi_{*} : H_{2n-2}(\widetilde{X},\R) \to H_{2n-2}(X,\mathbb{R})$ is the push-forward in homology. It is known that this definition is independent of the choice of $\pi$.
\end{defn}
Using this definition, for any classes $\alpha_1, \ldots, \alpha_{n-1} \in H^{1,1}_{BC}(X)$, we define
$$
c_1(\mathcal{E}) \cdot \alpha_1 \cdots \alpha_{n-1}
:= 
\underbrace{c_1(\mathcal{E})}_{\in H_{2n-2}(X)} \cdot 
\underbrace{\left(\delta^1(\alpha_1) \cdots \delta^1(\alpha_{n-1}) \right)}_{\in  H^{2n-2}(X, \mathbb{R}) \text{ by \eqref{eq-connect-h2}}} 
\in \R
$$

Let $f \colon Y \to X$ be a proper bimeromorphic morphism between normal analytic varieties. By \cite[Lemma 3.3]{HP16}, we can define a pullback map
$f^{*} \colon H^{1,1}_{BC}(X) \hookrightarrow H^{1,1}_{BC}(Y)$.
In this setting, we have
\begin{equation}
\label{eq-pullback-intersection}
c_1(\mathcal{E}) \cdot \alpha_1 \cdots \alpha_{n-1}
=
c_1(f^{[*]}\mathcal{E})\cdot f^{*}\alpha_1 \cdots f^{*}\alpha_{n-1}
\end{equation}
This follows directly from \cite[Lemma 2.2]{Ou25b} and from the fact that the following diagram commutes, even if $X$ and $Y$ do not necessarily have rational singularities:
$$
\xymatrix@C=25pt@R=20pt{
H^{1,1}_{BC}(X) \ar[r]^{f^{*}} \ar[d]_{\delta^1}
& H^{1,1}_{BC}(Y) \ar[d]^{\delta^1} \\
H^2(X, \mathbb{R}) \ar[r]^{f^{*}} & H^2(Y, \mathbb{R}) \\
}
$$

At the end of this subsection, we recall the definition of Fujiki's class.

\begin{defn}(\cite[Definition 2.2]{DH20})
\label{defn-Fujiki}
Let $X$ be a compact normal analytic variety. We say that $X$ belongs to \textit{Fujiki's class} if $X$ is bimeromorphic to a compact K\"ahler manifold.
\end{defn}
If $X$ admits a big class, then $X$ belongs to Fujiki's class. However, it remains unknown whether the converse holds when $X$ is singular (see \cite[Remark 2.3]{DH23}). Thus, this paper mainly treats varieties in Fujiki's class.

\section{Orbifold Chern classes}

In this section, we recall the definition of orbifold Chern classes and some elementary properties, such as the Hodge index theorem and the additivity of the second Chern class under exact sequences.

\subsection{Definition of orbifolds}
For the definition of orbifolds, we follow the convention of \cite[Subsection 3.1]{DO23}. We briefly recall a simplified version of the definition. For the precise definition, we also refer to \cite[Subsection 4.1]{IJZ25}.

As in \cite[Definition 3.1]{DO23}, a \emph{complex orbifold} \( Y \) of dimension \( n \) is a connected, second countable Hausdorff space equipped with an orbifold structure \( Y_{\orb} = \{ (U_i, G_i, \pi_i) \}_{i \in I} \) satisfying the following conditions:
\begin{enumerate}[label=$(\arabic*)$]
    \item Each \( U_i \) is an open subset of \( \mathbb{C}^n \),
    \item \( G_i \subset \mathrm{GL}_n(\mathbb{C}) \) is a finite group acting holomorphically on \( U_i \), and
    \item The map \( \pi_i \colon U_i \to U_i / G_i \) is the quotient map, such that \( U_i / G_i \cong Y_i \subset Y \) for some analytic open set \( Y_i \), and \( Y  =  \bigcup_{i \in I} Y_i\).
\end{enumerate}
These data are required to satisfy certain compatibility conditions as in \cite[Definition 4.1 (2)]{IJZ25}. The orbifold is called \emph{standard} if each \( G_i \) acts freely in codimension one. Throughout this paper, we always assume that the orbifold structure \( Y_{\orb} \) is effective (see \cite[Definition 4.1]{IJZ25}).

We often say that \( Y \) is the quotient space of a complex orbifold \(Y_\orb\). According to \cite[Remark~3.4 (1) and (2)]{DO23}, the quotient space \( Y \) of a complex orbifold \(Y_\orb\) is a normal analytic variety with quotient singularities, and conversely, any complex analytic variety with quotient singularities admits a unique standard orbifold structure.

An \emph{orbi-sheaf} $\mathcal{E}_{\orb}=\{\mathcal{E}_i\}_{i \in I}$ on a complex orbifold $Y_{\orb} = \{(U_i, G_i, \pi_i)\}_{i \in I}$ is a collection of holomorphic $G_i$-linearized sheaves $\mathcal{E}_i$ on $U_i$ satisfying certain compatibility conditions. The orbi-sheaf $\{\mathcal{E}_i\}_{i \in I}$ is said to be \emph{torsion-free} (resp.~\emph{reflexive}, \emph{locally free}) if each $\mathcal{E}_i$ is torsion-free (resp.~reflexive, locally free). As in the smooth case, when an orbi-sheaf \( \mathcal{E}_{\orb} \) is locally free, we identify it with the corresponding vector orbi-bundle \( E_{\orb} \).

Similarly, an \emph{orbifold differential form} $\sigma_{\orb}=\{\sigma_i\}_{i \in I}$ on a complex orbifold $Y_{\orb} = \{(U_i, G_i, \pi_i)\}_{i \in I}$ is a collection $\{\sigma_i\}$ of $G_i$-invariant differential forms on the charts $\{U_i\}$ satisfying certain compatibility conditions.

As in \cite[Subsection 2.2.2]{Kob14}, we can define the orbifold Chern classes of vector orbi-bundles. Let \( E_{\orb} = \{E_i\}_{i \in I} \) be a vector orbi-bundle over a complex orbifold \( Y_{\orb} \), and let \( h_{\orb} \) be a Hermitian metric on \( E_{\orb} \), given by a collection \( \{h_i\}_{i \in I} \) of \( G_i \)-invariant Hermitian metrics on the local bundles \( E_i \), compatible with the orbifold structure. The \emph{orbifold Chern class} \( c_{p}^{\orb}(E_{\orb}) \) is then defined via the \( p \)-th orbifold Chern forms \( \Theta_p := \{ \Theta_p(E_i, h_i) \}_{i \in I} \). By \cite{Sat56}, we have $c_{p}^{\orb}(E_{\orb})\in H^{2p}(Y, \mathbb{R})$.

\subsection{Definition of orbifold Chern classes}
\label{subsec-intersection-orbifold-nonpluripolar}

Let $X$ be a compact normal analytic variety with quotient singularities in codimension $2$.
We recall the orbifold Chern classes using orbifold modifications introduced in \cite[Section 2]{Ou25b}.

Let $\mathcal{E}$ be a reflexive sheaf. Since $X$ has at most quotient singularities in codimension $2$, by \cite[Theorem 1.2]{Ou24} and \cite[Theorem 1]{KO25}, there exists a bimeromorphic map $f: Y \to X$, where $Y$ has only quotient singularities and admits a standard orbifold structure $Y_{\orb} = \{ (V_i, G_i, \rho_i) \}_{i \in I}$. Define $\mathcal{F} := f^*\mathcal{E} / \Tor$ and set $\mathcal{F}_i := \rho_i^{*} \mathcal{F} / \tor$. Then the sheaf $\mathcal{F}$ defines a torsion-free orbi-sheaf $\mathcal{F}_{\orb} := \{ \mathcal{F}_i \}_{i \in I}$ on $Y_{\orb}$.

According to \cite[Theorem 3.10]{DO23}, there exists a functorial resolution $p_i: U_i \to V_i$ such that each $U_i$ is smooth and the sheaf $\mathcal{H}_i := p_i^* \mathcal{F}_i / \Tor$ is locally free. Moreover, by functoriality, each $U_i$ admits a $G_i$-action, which gives rise to a not necessarily standard orbifold structure $Z_{\orb} := \{ (U_i, G_i, \pi_i) \}_{i \in I}$. Let $Z$ be the quotient space associated with $Z_{\orb}$, and let $p: Z \to Y$ be the morphism induced by the collection $\{ p_i \}_{i \in I}$. Set $q := f \circ p: Z \to X$. Since $\mathcal{H}_{\orb} := \{ \mathcal{H}_i \}_{i \in I}$ is locally free, it defines an orbi-bundle.
Thus we obtain
$
c_2^{\orb}(\mathcal{H}_{\orb}) \in H^4(Z, \mathbb{R}).
$
\begin{defn}
For any classes $\alpha_1, \ldots, \alpha_{n-2} \in H^{1,1}_{BC}(X)$, we define the intersection number
$$
\widehat{c}_2(\mathcal{E}) \cdot  \alpha_1 \cdots \alpha_{n-2}  
:= \underbrace{c_2^{\orb}(\mathcal{H}_{\orb}) }_{\in H^4(Z, \mathbb{R})}\cdot 
\underbrace{\delta^1(q^* \alpha_1) \cdots \delta^1(q^* \alpha_{n-2})}_{ \in H^{2n-4}(Z, \mathbb{R})}
\in \R
$$
Here we use the morphism $\delta^1 : H^{1,1}_{\mathrm{BC}}(Z) \to H^2(Z, \R)$ as in \eqref{eq-connect-h2}, and the above product is the cup product in singular cohomology.
\end{defn}
By \cite[Lemma 2.2]{Ou25b}, the above intersection number is independent of the choice of the orbifold modification $q : Z \to X$.
In the same way, one can define
$$
\widehat{c}_1(\mathcal{E})^2 \cdot  \alpha_1 \cdots \alpha_{n-2} 
\quad \text{and} \quad \widehat{c}_1(\mathcal{E})\cdot \widehat{c}_1(\mathcal{E}') \cdot \alpha_1 \cdots \alpha_{n-2}
$$
for any reflexive sheaves $\mathcal{E}$ and $\mathcal{E}'$ on $X$. The same construction also defines $\widehat{c}_1(\mathcal{E}) \cdot  \alpha_1 \cdots \alpha_{n-1}$. Since $\mathcal{E}$ is reflexive, this coincides with $c_1(\mathcal{E}) \cdot  \alpha_1 \cdots \alpha_{n-1}$ in Definition~\ref{defn-first-chern}.
We introduce the following notation:
\begin{defn}
\label{defn-Bogomolov-discriminant}
Let $X$ be a compact normal analytic variety with quotient singularities in codimension $2$. We define
\[
\widehat{c}_2(X) := \widehat{c}_2(\Omega_X^{[1]})
\quad \text{and} \quad
\widehat{c}_1(X)^2 := \widehat{c}_1(\Omega_X^{[1]})^2.
\]
For a rank $r$ reflexive sheaf $\mathcal{E}$ on $X$, the \emph{Bogomolov discriminant} is defined by
\[
\widehat{\Delta}(\mathcal{E}) := 2r \widehat{c}_2(\mathcal{E}) - (r-1) \widehat{c}_1(\mathcal{E})^2.
\]
\end{defn}

\begin{rem}
According to \cite[Lemma 5.8]{GK20}, any klt variety has quotient singularities in codimension two, thus we can define the orbifold second Chern class.
\end{rem}

\subsection{Hodge index theorem}
We begin by recalling the following Hodge index type theorem.

\begin{lem}$($\cite[Lemmas 3.6–3.8, 3.10]{Zh24}$)$
\label{lem-zhang-hodge-index}
Let $X$ be an $n$-dimensional compact K\"ahler manifold, and let $\alpha_1, \ldots, \alpha_{n-2} \in H^{1,1}(X,\mathbb{R})$ be nef and big classes. Set $\Omega := \alpha_1 \cdots \alpha_{n-2}$. Suppose $\alpha \in H^{1,1}(X,\mathbb{R})$ is a nef class satisfying $\alpha^2 \cdot \Omega > 0$.

For any $\beta, \gamma \in H^{1,1}(X,\mathbb{C})$, define
$$
\mathrm{Q}_{\Omega}(\beta, \gamma) := \beta \cdot \bar{\gamma} \cdot \Omega,
\quad \text{and} \quad
\mathrm{P}_{\alpha \cdot \Omega} := \{ \gamma \in H^{1,1}(X,\mathbb{C}) \mid \gamma \cdot (\alpha \cdot \Omega) = 0 \}.
$$
Then the following hold:
\begin{enumerate}[label=$(\arabic*)$]
    \item For any $\beta \in H^{1,1}(X,\mathbb{C})$, there exist $a \in \mathbb{C}$ and $\beta_0 \in \mathrm{P}_{\alpha \cdot \Omega}$ such that $\beta = a \cdot \alpha + \beta_0$.
    \item $\mathrm{Q}_{\Omega}$ is semi-negative definite on $\mathrm{P}_{\alpha \cdot \Omega}$.
    \item For any $\gamma \in \mathrm{P}_{\alpha \cdot \Omega}$, $\mathrm{Q}_{\Omega}(\gamma,\gamma) = 0$ if and only if $\gamma \cdot \Omega = 0$.
\end{enumerate}
\end{lem}

We now state the following generalization of the Hodge index theorem for analytic varieties:

\begin{lem}
\label{lem-orbifold-Hodge-index}
Let $X$ be an $n$-dimensional compact normal analytic variety in Fujiki's class. Let $\beta \in H^{1,1}_{BC}(X)$ be a nef class, and let $\alpha_1, \ldots, \alpha_{n-2} \in H^{1,1}_{BC}(X)$ be nef and big classes. Let $\gamma \in H^{1,1}_{BC}(X)$ and set $\Omega := \alpha_1 \cdots \alpha_{n-2}$.
\begin{enumerate}[label=$(\arabic*)$]
    \item If $\beta^2 \cdot \Omega > 0$, then the following inequality holds:
    \begin{equation}
    \label{equa-hodge-1}
    (\beta^2 \cdot \Omega) \cdot (\gamma^2 \cdot \Omega) \leq (\beta \cdot \gamma \cdot \Omega)^2.
    \end{equation}
    
    \item If $\beta^2 \cdot \Omega = 0$, $\beta \cdot \gamma \cdot \Omega = 0$, and there exists a nef and big class $\alpha \in H^{1,1}_{BC}(X)$ such that $\beta \cdot \alpha \cdot \Omega > 0$, then we have:
    \begin{equation}
    \label{equa-hodge-2}
    \gamma^2 \cdot \Omega \leq 0.
    \end{equation}
    
    \item In either case $(1)$ or $(2)$, if equality holds in \eqref{equa-hodge-1} or \eqref{equa-hodge-2}, then there exists a real number $\lambda$ such that
    $
    (\gamma - \lambda \beta) \cdot \delta \cdot \Omega = 0
    $
    for every $\delta \in H^{1,1}_{BC}(X)$.
\end{enumerate}
\end{lem}

\begin{proof}
By taking a resolution, it suffices to prove the result under the assumption that $X$ is a compact K\"ahler manifold.
We use the notation from Lemma~\ref{lem-zhang-hodge-index}.

(1) By Lemma~\ref{lem-zhang-hodge-index} (1), we can write $\gamma = a \cdot \beta + \gamma_0$ for some $\gamma_0 \in \mathrm{P}_{\beta \cdot \Omega}$, where $a = \frac{\mathrm{Q}_{\Omega}(\gamma,\beta)}{\mathrm{Q}_{\Omega}(\beta,\beta)} \in \mathbb{R}$. Then by Lemma~\ref{lem-zhang-hodge-index} (2),
$$
\mathrm{Q}_{\Omega}
\left(\gamma - a \cdot \beta, \gamma - a \cdot \beta\right) \leq 0,
$$
which is equivalent to \eqref{equa-hodge-1}. Moreover, equality holds if and only if $(\gamma - a \cdot \beta) \cdot \Omega = 0$ by Lemma~\ref{lem-zhang-hodge-index} (3).

(2) Write $\beta = b \cdot \alpha + \beta_0$ and $\gamma = c \cdot \alpha + \gamma_0$ for some $b, c \in \mathbb{R}$ and $\beta_0, \gamma_0 \in \mathrm{P}_{\alpha \cdot \Omega}$. Since $\beta \cdot \alpha \cdot \Omega > 0$, we have $b > 0$, and hence $\gamma - \frac{c}{b} \beta \in \mathrm{P}_{\alpha \cdot \Omega}$. Therefore,
$$
\gamma^2 \cdot \Omega = \mathrm{Q}_{\Omega} \left( \gamma - \frac{c}{b} \beta, \gamma - \frac{c}{b} \beta \right) \leq 0.
$$
Equality holds if and only if $(\gamma - \frac{c}{b} \beta) \cdot \Omega = 0$. This completes the proof.
\end{proof}

\begin{rem}
\label{rem-Hodge-index-num}
Under the same assumptions as in Lemma~\ref{lem-orbifold-Hodge-index}, the condition ``$\beta \cdot \gamma \cdot \Omega = 0$ for all nef and big classes $\beta$'' is equivalent to ``$\gamma \cdot \Omega \equiv 0$'', that is,
$$
\delta \cdot \gamma \cdot \Omega = 0
\quad \text{for all } \delta \in H^{1,1}_{BC}(X).
$$
\end{rem}

Using Lemma~\ref{lem-orbifold-Hodge-index}, we obtain the following Hodge index theorem for orbifold Chern classes.

\begin{prop}
\label{prop-Hodge-index}
Let $X$ be a compact normal analytic variety in Fujiki's class with quotient singularities in codimension $2$. Let $\beta \in H^{1,1}_{BC}(X)$ be a nef class, and let $\alpha_1, \dots, \alpha_{n-2} \in H^{1,1}_{BC}(X)$ be nef and big classes. Set $\Omega := \alpha_1 \cdots \alpha_{n-2}$.

For any reflexive sheaf $\mathcal{E}$ on $X$, the following statements hold:
\begin{enumerate}[label=$(\arabic*)$]
    \item If $\beta^2 \cdot \Omega > 0$, then
    \begin{equation}
    \label{eq-hodge-index-ob-1}
    (\widehat{c}_1(\mathcal{E})^2 \cdot \Omega) \cdot
        (\beta^2 \cdot \Omega)\leq (\widehat{c}_1(\mathcal{E}) \cdot \beta \cdot \Omega)^2.
    \end{equation}
    \item If $\beta^2 \cdot \Omega = 0$, $ \widehat{c}_1(\mathcal{E})  \cdot\beta\cdot \Omega = 0$, and there exists a nef and big class $\alpha \in H^{1,1}_{BC}(X)$ such that $ \beta \cdot \alpha \cdot\Omega > 0$, then
    \begin{equation}
    \label{eq-hodge-index-ob-2}
        \widehat{c}_1(\mathcal{E})^2 \cdot \Omega \leq 0.
    \end{equation}
    \item In either case $(1)$ or $(2)$, if equality holds in \eqref{eq-hodge-index-ob-1} or \eqref{eq-hodge-index-ob-2}, then there exists a real number $\lambda$ such that
    \(
    \widehat{c}_1(\mathcal{E}) \cdot \delta \cdot \Omega
    = \lambda \beta \cdot \delta \cdot \Omega
    \)
    for every $\delta \in H^{1,1}_{BC}(X)$.
\end{enumerate}
\end{prop}

\begin{proof}
As in Subsection~\ref{subsec-intersection-orbifold-nonpluripolar}, let $q \colon Z \to X$ be a bimeromorphic morphism such that $Z$ admits an orbifold structure $Z_{\orb}$, and let $E_{\orb}$ be an orbi-bundle on $Z_{\orb}$ satisfying
\[
\widehat{c}_1(\mathcal{E})^2 \cdot \polar = c_1^{\orb}(E_{\orb})^2 \cdot q^*\polar 
\quad \text{and} \quad
\widehat{c}_1(\mathcal{E}) \cdot \beta \cdot \polar = c_1^{\orb}(E_{\orb}) \cdot q^*\beta \cdot q^*\polar.
\]
Therefore, the result follows by applying Lemma~\ref{lem-orbifold-Hodge-index} to $\gamma := c_1^{\orb}(E_{\orb})$ on $Z$.
\end{proof}

\subsection{Exact sequence}
In this subsection, we study \cite[Lemma 2.3]{Kaw92} and \cite[Lemma~2.13]{IMM24} in the context of analytic varieties, examining the behavior of the second Chern class under an exact sequence. 
To this end, we need the following orbifold version of the Grothendieck--Riemann--Roch formula established in \cite{ma2025superconnection}.

\begin{thm}\cite[Theorem~1.1, Theorem~1.2, Definition~8.1]{ma2025superconnection}
\label{thm-coherent-chernclass}
Let \( X_\orb \) be a compact complex orbifold. Then there exists a unique group homomorphism
\[
\mathrm{ch}_{BC}: K(X_\orb) \rightarrow H_{BC}^{*}(IX_\orb,\C)
\]
from the Grothendieck group of coherent orbi-sheaves on \( X_\orb \) to the Bott--Chern cohomology of the inertia orbifold \( IX_\orb \) of $X_\orb$, satisfying the following properties:
\begin{enumerate}[label=$(\arabic*)$]
    \item For any orbi-bundle \( E_\orb \), the class \( \mathrm{ch}_{BC}(E_\orb) \) coincides with the definition given in \cite[Section~1.2]{ma2005orbifolds}.
    \item The map \( \mathrm{ch}_{BC} \) is functorial with respect to pullbacks of orbifolds.
    \item The map \( \mathrm{ch}_{BC} \) satisfies the Grothendieck--Riemann--Roch formula for orbifold embeddings: for any embedding \( i_\orb: Z_\orb \hookrightarrow Y_\orb \) of compact complex orbifolds, and any coherent orbi-sheaf \( \mathcal{F}_\orb \) on \( Z_\orb \), we have
    \[
    \mathrm{ch}_{BC}((i_\orb)_*\mathcal{F}_\orb) = (i_\orb)_* \left( \frac{\mathrm{ch}_{BC}(\mathcal{F}_\orb)}{\mathrm{Td}_{BC}(N_{Z_\orb/Y_\orb})} \right) \quad \text{in } H_{BC}^*(IY_{\orb},\C),
    \]
    where \( \mathrm{Td}_{BC} \) denotes the Bott--Chern Todd class.
\end{enumerate}
\end{thm}

\begin{rem}
Although Theorem~\ref{thm-coherent-chernclass} is stated in terms of the Bott--Chern cohomology of the inertia orbifold ${IX}_\orb$ of $X_\orb$, it can indeed be applied to $X_\orb$. We first explain the definition of the Chern character for an orbi-bundle $E_\orb$ given in \cite[Section 1.2]{ma2005orbifolds}.
Following the notation in \cite[Page~2211]{ma2005orbifolds}, we write $IX_\orb=X_\orb\sqcup \Sigma X_\orb$. Then $\mathrm{ch}^\orb(E_\orb)$ can be regarded as the projection of $\mathrm{ch}_{BC}(E_\orb)$ from $H_{BC}^*(IX_\orb,\C)$ to $H_{BC}^*(X_\orb,\C)$. Similarly, the restriction of the induced morphism between inertia orbifolds to the original orbifold coincides with the original morphism (see \cite[Remark 2.56]{ma2025superconnection}).
\end{rem}

We now use this fact to establish the following statement.
\begin{prop}[{cf.~\cite[Lemma 2.3]{Kaw92}, \cite[Lemma~2.13]{IMM24}}]
\label{prop-exact-sequence}
Let \( X \) be a compact normal analytic variety with quotient singularities in codimension \( 2 \). Let \( \alpha_1, \ldots, \alpha_{n-2} \in H^{1,1}_{BC}(X) \) be nef classes, and set \( \polar := \alpha_1 \cdots \alpha_{n-2} \).

Consider an exact sequence of torsion-free sheaves
\[
0 \to \mathcal{E} \to \mathcal{F} \to \mathcal{G} \to 0,
\]
where \( \mathcal{E} \) and \( \mathcal{F} \) are reflexive. Then the following statements hold:
\begin{align}
\label{eq-1stchern-exact}
&\widehat{c}_1(\mathcal{F})^2 \cdot \polar 
= \left( \widehat{c}_1(\mathcal{E})^2 + \widehat{c}_1(\mathcal{G}^{\vee\vee})^2 + 2\,\widehat{c}_1(\mathcal{E}) \cdot \widehat{c}_1(\mathcal{G}^{\vee\vee}) \right) \cdot \polar. \\
\label{eq-2ndchern-exact}
&\widehat{c}_2(\mathcal{F}) \cdot \polar 
\ge \left( 
\widehat{c}_2(\mathcal{E}) 
+ \widehat{c}_2(\mathcal{G}^{\vee\vee}) 
+ \widehat{c}_1(\mathcal{E}) \cdot \widehat{c}_1(\mathcal{G}^{\vee\vee}) 
\right) \cdot \polar.
\end{align}
\end{prop}

\begin{proof}
Let \( f:Y \to X \) be the orbifold modification as in \cite[Theorem 1.2]{Ou24} and \cite[Theorem 1]{KO25}, where \( Y \) admits a standard orbifold structure \( Y_{\orb} = \{(V_i, H_i, \rho_i)\}_{i \in I} \), and there exists a Zariski open subset \( X^{\circ} \subset X \) such that \( \codim (X \setminus X^{\circ}) \ge 3 \) and \( f \) is an isomorphism over \( X^{\circ} \). Set \( f_i := f \circ \rho_i : V_i \to X \). We define the reflexive orbi-sheaf \( \mathcal{E}_{\orb} := \{f_i^{[*]}\mathcal{E}\}_{i \in I} \) on \( Y_{\orb} \). By \cite[Lemma 4.9]{ZZZ25}, we obtain
\begin{equation*}
\widehat{c}_2(\mathcal{E}) \cdot \polar = c_{2}^{\orb}(\mathcal{E}_{\orb}) \cdot f^{*}\polar.
\end{equation*}
Thus, we may work with orbi-sheaves on \( Y_{\orb} \). Define \( \mathcal{F}_{\orb} := \{f_i^{[*]}\mathcal{F}\}_{i \in I} \) and \( \mathcal{G}_{\orb} := \{f_i^{[*]}(\mathcal{G}^{\vee\vee})\}_{i \in I} \) on \( Y_{\orb} \). Then we have the following exact sequence on \( Y_{\orb} \):
\begin{equation}
\label{eq-IMM24-1}
0 \to \mathcal{E}_{\orb} \to \mathcal{F}_{\orb} \overset{\phi_{\orb}}{\to} \mathcal{G}_{\orb} \to \mathcal{T}_{\orb} \to 0,
\end{equation}
where \( \mathcal{T}_{\orb} := \Coker(\phi_{\orb}) \) is a torsion orbi-sheaf. Since each \( f_i \) is \'{e}tale in codimension one, the map \( \mathcal{E}_{\orb} \to \mathcal{F}_{\orb} \) is injective and \( \codim(\supp(\mathcal{T}_{\orb}) \cap f^{-1}(X \setminus X^{\circ})) \ge 2 \). Therefore, we have \( c_1^{\orb}(\mathcal{T}_{\orb})|_{f^{-1}(X^{\circ})} = 0 \), which yields
\[
c_{1}^{\orb}(\mathcal{F}_{\orb})^2 \cdot f^*\Omega 
= \big(c_{1}^{\orb}(\mathcal{E}_{\orb})^2 + c_{1}^{\orb}(\mathcal{G}_{\orb})^2 + 2\,c_1^{\orb}(\mathcal{E}_{\orb}) \cdot c_1^{\orb}(\mathcal{G}_{\orb})\big) \cdot f^*\Omega,
\]
by \cite[Lemma 2.2]{Ou25b}. Hence, \eqref{eq-1stchern-exact} follows.

We now prove \eqref{eq-2ndchern-exact}. By \eqref{eq-IMM24-1}, we have
\begin{equation*}
c_2^\orb(\mathcal{F}_\orb) \cdot f^*\Omega 
= \big(c_2^\orb(\mathcal{E}_\orb) + c_2^\orb(\mathcal{G}_\orb) + c_1^\orb(\mathcal{E}_\orb) \cdot c_1^\orb(\mathcal{G}_\orb) - c_2^\orb(\mathcal{T}_\orb)\big) \cdot f^*\Omega.
\end{equation*}
Therefore, it suffices to show that \( c_2^\orb(\mathcal{T}_\orb) \cdot f^*\Omega \leq 0 \). The argument is similar to that in \cite[Theorem 7.9]{CHP16}. Let \( S_\orb \) be the component of \( \supp(\mathcal{T}_\orb) \) of pure dimension \( n-2 \). We may assume that all irreducible components \( S_{1,\orb}, \ldots, S_{m,\orb} \) of \( S_\orb \) are smooth, after taking an orbifold resolution of singularities if necessary. Set \( \mathcal{Q}_\orb := (i_{S_\orb})_* (\mathcal{T}_\orb|_{S_\orb}) \), where \(i_{S_\orb}: S_\orb \to Y_\orb \) is the natural inclusion.

Since \( c_2^\orb(\mathcal{T}_\orb)|_{f^{-1}(X^{\circ})} = c_2^\orb(\mathcal{Q}_\orb)|_{f^{-1}(X^{\circ})} \), it follows from \cite[Lemma 2.2]{Ou25b} that
\[
c_2^\orb(\mathcal{T}_\orb) \cdot f^*\polar = c_2^\orb(\mathcal{Q}_\orb) \cdot f^*\polar.
\]
Thus, it is enough to show that \( c_2^\orb(\mathcal{Q}_\orb) \cdot f^*\Omega \leq 0 \). Since $\mathcal{Q}_\orb$ is supported in codimension two, Theorem~\ref{thm-coherent-chernclass}~(3) implies that
\[
c_2^\orb(\mathcal{Q}_\orb) = -\sum_{l=1}^{m} a_l  \{S_{l,\orb}\}
\]
for some coefficients $a_l \geq 0$. (See also \cite[Remark 8]{Cao13}.)
As \( \alpha_1, \ldots, \alpha_{n-2} \) are nef, we conclude that \( c_2^\orb(\mathcal{Q}_\orb) \cdot f^*\Omega \leq 0 \), completing the proof.
\end{proof}

\section{Miyaoka's inequality and proofs of main results}
\label{sec-Miyaoka's-inequality}
In this section, we aim to extend Miyaoka's inequality to analytic varieties in Fujiki's class. We prove Theorems~\ref{thm-Miyaoka-ineq-nonuniruled}--\ref{thm-nef-anti-canonical} and Corollaries~\ref{cor-MY-canonical}--\ref{cor-MY-anti-canonical}.

\subsection{Generic nefness theorem}
\label{subsec-generic-neffness-thm}
The \textit{generic nefness theorem} was originally proved by Miyaoka \cite{Miy87}. It states that if a normal projective variety $X$ is non-uniruled, then $\mu^{\min}_{H_1 \cdots H_{n-1}}(\Omega_X^{[1]}) \ge 0$ holds for any ample divisors $H_1, \ldots, H_{n-1}$. We will define the minimum slope $\mu_{H_1 \cdots H_{n-1}}^{\min}(\Omega_{X}^{[1]})$ later. This result has been extended to the setting where $X$ is a klt K\"ahler variety as follows.

\begin{thm}$($\cite[Theorem 1.4]{Eno88}, \cite[Theorem 1.2]{Cao13}, \cite[Theorem 4.1]{Gue16}$)$
\label{thm-Enoki-genericnef}
Let $X$ be a compact klt K\"ahler variety. If $K_X$ $($resp.~$-K_X$$)$ is nef, then $\mu_{\{\omega\}^{n-1}}^{\min}(\Omega_{X}^{[1]}) \ge 0$ $($resp.~$\mu_{\{\omega\}^{n-1}}^{\min}(\mathcal{T}_{X}) \ge 0$$)$ for any K\"ahler form $\omega$.
\end{thm}

We now consider the case where $X$ belongs to Fujiki's class and the slope is defined with respect to $\alpha_1 \cdots \alpha_{n-1}$, as described in the following setup.

\begin{setup}
\label{setup-genericnef}
Let $X$ be a compact normal analytic variety in Fujiki's class, and let 
$\alpha_1, \ldots, \alpha_{n-1}\in H^{1,1}_{\mathrm{BC}}(X)$ be nef classes. Set
$$
\gamma := 
\delta^1(\alpha_1) \cdots \delta^1(\alpha_{n-1}) \in H^{2n-2}(X, \R).
$$
Here we use the morphism $\delta^1 : H^{1,1}_{\mathrm{BC}}(X) \to H^2(X, \R)$ as in \eqref{eq-connect-h2}, and the above product is the cup product in singular cohomology.
\end{setup}

Under Setup~\ref{setup-genericnef}, for any effective divisor $D$, we have $c_1(\mathcal{O}_X(D))\cdot \gamma \ge 0$ since all $\alpha_i$ are nef. Thus, for any torsion-free sheaves $\mathcal{F} \subset \mathcal{G}$ of the same rank, we have $c_1(\mathcal{F}) \cdot \gamma \le c_1(\mathcal{G}) \cdot \gamma$. Therefore, we can prove the following properties by the same argument as in \cite[Section 2]{GKP16a}, \cite[Subsection 3.2]{DO23}, and \cite[Subsubsection 2.2.3]{ZZZ25}. 
In the following statement, we assume that $\mathcal{E}$ is a torsion-free sheaf on $X$, and for every subsheaf $0 \neq \mathcal{F} \subseteq \mathcal{E}$, we define
\(
\mu_{\gamma}(\mathcal{F}):=\frac{c_1(\mathcal{F})\cdot\gamma}{\rk(\mathcal{F})}
\).

\begin{enumerate}
\item The maximum slope of $\mathcal{E}$ is defined by
$$
\mu^{\max}_{\gamma}(\mathcal{E}) := \sup \left\{ \mu_{\gamma}(\mathcal{F}) \mid 0 \neq \mathcal{F} \subseteq \mathcal{E}, \, \mathcal{F} \text{ torsion-free} \right\}.
$$
That is, the supremum is taken over all non-zero torsion-free subsheaves $\mathcal{F} \subseteq \mathcal{E}$. Then $\mu^{\max}_{\gamma}(\mathcal{E})$ is finite. Moreover, $\mathcal{E}$ is $\gamma$-semistable if and only if $\mu^{\max}_{\gamma}(\mathcal{E}) = \mu_{\gamma}(\mathcal{E})$.

\item There exists a saturated subsheaf $\mathcal{E}_{\max} \subset \mathcal{E}$ such that $\mu^{\max}_{\gamma}(\mathcal{E}) = \mu_{\gamma}(\mathcal{E}_{\max})$. This subsheaf $\mathcal{E}_{\max}$ is called the \textit{maximal destabilizing subsheaf}, and it is $\gamma$-semistable.

\item The minimum slope is defined by
$$
\mu^{\min}_{\gamma}(\mathcal{E}) := \inf \left\{ \mu_{\gamma}(\mathcal{Q}) \mid \mathcal{E} \twoheadrightarrow \mathcal{Q}, \, \mathcal{Q} \text{ non-zero torsion-free} \right\}.
$$
That is, the infimum is taken over all non-zero torsion-free quotient sheaves $\mathcal{Q}$ of $\mathcal{E}$. Then we have $\mu^{\max}_{\gamma}(\mathcal{E}^\vee) = -\mu^{\min}_{\gamma}(\mathcal{E})$. We often say that $\mathcal{E}$ is \textit{$\gamma$-generically nef} if $\mu^{\min}_{\gamma}(\mathcal{E}) \ge 0$.

\item Every torsion-free sheaf admits a Harder--Narasimhan filtration, and every semistable torsion-free sheaf admits a Jordan--Hölder filtration.
\end{enumerate}

\begin{prop}
\label{prop-CP-foliation}
$[${cf.~\cite[Theorem~1.4]{CP19},
\cite[Proposition~2.2]{Ou17},
\cite[Theorem~1.4]{Ou25},
\cite[Corollary~1.5]{CP25}}$]$
Under Setup~\ref{setup-genericnef}, let \( \mathcal{F} \subset \mathcal{T}_X \) be a saturated subsheaf on \( X \). Assume that one of the following holds:
\begin{enumerate}[label=$(\arabic*)$]
    \item \( \mu^{\min}_{\gamma}(\mathcal{F}) > 0 \) and \( 2\mu^{\min}_{\gamma}(\mathcal{F}) > \mu^{\max}_{\gamma}(\mathcal{T}_X / \mathcal{F}) \).
    \item \( \mu^{\min}_{\gamma}(\mathcal{F}) > 0 \) and \( \mu_{\gamma}(\mathcal{F}) = \mu^{\max}_{\gamma}(\mathcal{T}_X) \).
\end{enumerate}
Then \( \mathcal{F} \) is an algebraic foliation with rationally connected leaves.
\end{prop}

\begin{proof}
The proof is similar to that of \cite[Proposition 2.2]{Ou17}.
Consider case (1). Let \( \pi: \widetilde{X} \to X \) be a resolution such that \( \widetilde{X} \) is a compact K\"ahler manifold. Let \( \widetilde{\mathcal{F}} \subset \mathcal{T}_{\widetilde{X}} \) be the saturated subsheaf induced by \( \mathcal{F} \). Then \( \widetilde{\mathcal{F}} \) and \( \pi^{[*]} \mathcal{F} \) are isomorphic over \( \pi^{-1}(X_{\reg}) \). Hence, we have
$$
\mu_{\gamma}(\mathcal{F})  
\underalign{\eqref{eq-pullback-intersection}}{=}
\mu_{\pi^{*}\gamma}(\pi^{[*]} \mathcal{F}) = 
\mu_{\pi^{*}\gamma}(\widetilde{\mathcal{F}}).
$$
Similarly, we obtain
$$
\mu^{\min}_{\pi^{*}\gamma}(\widetilde{\mathcal{F}}) 
= \mu^{\min}_{\gamma}(\mathcal{F}) \quad \text{and} \quad
\mu^{\max}_{\pi^{*}\gamma}(\mathcal{T}_{\widetilde{X}}/\widetilde{\mathcal{F}}) 
= \mu^{\max}_{\gamma}(\mathcal{T}_X/\mathcal{F}).
$$
Thus, we have \( 2\mu^{\min}_{\pi^{*}\gamma}(\widetilde{\mathcal{F}}) > \mu^{\max}_{\pi^{*}\gamma}(\mathcal{T}_{\widetilde{X}}/\widetilde{\mathcal{F}}) \), and hence \( \widetilde{\mathcal{F}} \) is a foliation. Note that the class \( \pi^{*}\gamma \) is movable; that is, \( \pi^{*}\gamma \cdot \beta \ge 0 \) for every pseudo-effective (in short, psef) class \( \beta \) on $\widetilde{X}$.
Therefore, by applying \cite[Theorem 1.4]{Ou25} and \cite[Corollary 1.5]{CP25}, \( \widetilde{\mathcal{F}} \) is an algebraic foliation with rationally connected leaves. Since \( \mathcal{F} \cong (\pi_{*} \widetilde{\mathcal{F}})^{\vee\vee} \), the same conclusion holds for \( \mathcal{F} \).

Now consider case (2). In this case, we claim that \( 2\mu^{\min}_{\gamma}(\mathcal{F}) > \mu^{\max}_{\gamma}(\mathcal{T}_X / \mathcal{F}) \).
Indeed, from the assumption \( \mu_{\gamma}(\mathcal{F}) = \mu^{\max}_{\gamma}(\mathcal{T}_X) \), we obtain
$$
\mu^{\min}_{\gamma}(\mathcal{F}) = \mu_{\gamma}(\mathcal{F}) 
\ge \mu^{\max}_{\gamma}(\mathcal{T}_X / \mathcal{F}).
$$
(See \cite[Remark 1.16]{Cla16}.) Since \( \mu^{\min}_{\gamma}(\mathcal{F}) > 0 \), it follows immediately that \( 2 \mu^{\min}_{\gamma}(\mathcal{F}) > \mu^{\max}_{\gamma}(\mathcal{T}_X / \mathcal{F}) \), and we are reduced to case (1).
\end{proof}

\begin{prop}[{cf.~\cite[Corollary 6.4]{Miy87}}]
\label{prop-genericnef--nonuniruled}
Under Setup~\ref{setup-genericnef}, assume one of the following:
\begin{enumerate}[label=$(\arabic*)$]
    \item $X$ is non-uniruled.
    \item $X$ has at most canonical singularities and $K_X$ is psef.
\end{enumerate}
Then we have $\mu^{\min}_{\gamma}(\Omega_X^{[1]}) \ge 0$.
\end{prop}

\begin{proof}
(1) This follows from the same argument as in \cite[Corollary 6.4]{Miy87}. Assume, for contradiction, that $\mu^{\min}_{\gamma}(\Omega_X^{[1]}) < 0$. From $\mu^{\max}_{\gamma}(\mathcal{T}_X) > 0$,  there exists a maximal destabilizing subsheaf $\mathcal{E}_{\max} \subset \mathcal{T}_X$ with $\mu^{\min}_{\gamma}(\mathcal{E}_{\max}) > 0$.
This sheaf satisfies condition (2) in Proposition~\ref{prop-CP-foliation}, so $\mathcal{E}_{\max}$ defines an algebraic foliation with rationally connected leaves, contradicting the assumption that $X$ is non-uniruled.

(2) Under this assumption, $X$ is also non-uniruled. Indeed, let $\pi: \widetilde{X} \to X$ be a resolution where $\widetilde{X}$ is a compact K\"ahler manifold. Since $X$ has canonical singularities, we can write $K_{\widetilde{X}} = \pi^{*}K_X + E$ for some effective $\pi$-exceptional divisor $E$. Hence $K_{\widetilde{X}}$ is psef, and by \cite[Theorem 1.1]{Ou25}, $\widetilde{X}$ is non-uniruled. It follows that $X$ is non-uniruled as well.
\end{proof}

Note that even if $X$ is klt and $K_X$ is psef, $X$ is not necessarily non-uniruled (see, for example, the Ueno--Campana threefold in \cite{Cam11}). Thus, we do not know whether the generic nefness theorem holds when $X$ is klt and $K_X$ is psef.

Next, we consider the case where \( -K_X \) is nef.

\begin{prop}[{cf.~\cite[Theorem 1.4]{Ou17}, \cite[Theorem 5.2]{MWWZ25}}]
\label{prop-genericnef--antinef}
Under Setup~\ref{setup-genericnef}, if $-K_X$ is nef and $X$ is klt, then $\mu^{\min}_{\gamma}(\mathcal{T}_X) \ge 0$.
\end{prop}

\begin{proof}

The proof follows \cite[Theorem 1.4]{Ou17} with slight modifications. Assume, for contradiction, that $\mu_{\gamma}^{\min}(\mathcal{T}_X) < 0$. Then, by the argument of \cite[Theorem 1.4]{Ou17}, there exists a saturated subsheaf $\mathcal{F} \subset \mathcal{T}_X$ satisfying condition (1) in Proposition~\ref{prop-CP-foliation} such that
\begin{equation}
\label{eq-negative-beta}
\mu_{\gamma}(\mathcal{T}_X / \mathcal{F}) < 0.
\end{equation}
Therefore, by Proposition~\ref{prop-CP-foliation}, $\mathcal{F}$ defines an algebraic foliation with rationally connected leaves and induces a meromorphic map $f: X \dashrightarrow Y$ whose general fiber is rationally connected. Let $\pi : \widetilde{X} \to X$ be a resolution and $\widetilde{f} : \widetilde{X} \to Y$ a morphism resolving $f$. By taking a Stein factorization, we may assume that $\widetilde{f}$ has connected fibers. Since $X$ lies in Fujiki's class, both $\widetilde{X}$ and $Y$ can be assumed to be K\"ahler.
$$
\xymatrix@C=40pt@R=30pt{
& \widetilde{X} \ar[ld]_{\pi} \ar[rd]^{\widetilde{f}} & \\
X \ar@{-->}[rr]^f && Y
}
$$
As in \cite[Lemma 5.12]{CP19} (see also \cite[Theorem 5.1]{CP25} and \cite[Proposition 1.4]{Cla16}), possibly after taking modifications of $\widetilde{X}$ and $Y$, the foliation $\widetilde{\mathcal{F}}$ induced by $\widetilde{f}$ coincides with $\pi^{[*]} \mathcal{F}$ outside the $\pi$-exceptional locus, and
\begin{equation}
\label{eq-foliation}
K_{\widetilde{\mathcal{F}}} \sim K_{\widetilde{X}/Y} - \mathrm{Ram}(\widetilde{f}) + E,
\end{equation}
holds for some $\pi$-exceptional divisor $E$. Here, $K_{\widetilde{\mathcal{F}}}$ denotes the Cartier divisor associated with $\det(\widetilde{\mathcal{F}}^{\vee})$, and $\mathrm{Ram}(\widetilde{f})$ denotes the ramification divisor of $\widetilde{f}$, whose precise definition will be given later in Lemma~\ref{lem-CP-nefcase}. 
Thus we obtain
\begin{equation}
\label{eq-tangent-nef-1}
\left(K_{\widetilde{\mathcal{F}}} - \pi^{*}K_X\right) \cdot
\pi^{*}\gamma
\underalign{\eqref{eq-pullback-intersection}}{=}
(c_1(-K_X) - c_1(\mathcal{F})) \cdot \gamma 
=\rk(\mathcal{T}_X/\mathcal{F}) \cdot \mu_{\gamma}(\mathcal{T}_X / \mathcal{F})
\underalign{(\ref{eq-negative-beta})}{<} 0.
\end{equation}

Since $X$ is klt, there exist $\pi$-exceptional effective divisors $G$ and $\Delta$ such that $(\widetilde{X}, \Delta)$ is klt and
$$
K_{\widetilde{X}} + \Delta \sim \pi^{*}K_X + G.
$$
We now verify the following:
\begin{itemize}
\item The divisor $-\pi^{*}K_X$ is nef, and $\Delta$ is psef. Since $(\widetilde{X}, \Delta)$ is klt, the $\Q$-line bundle $\mathcal{O}_{\widetilde{X}}(\Delta)$ admits a singular Hermitian metric $h_{\Delta}$ with trivial multiplier ideal sheaf.
\item The restriction $(K_{\widetilde{X}} - \pi^{*}K_X + \Delta)|_{X_y} \sim G|_{X_y}$ is psef for general $y \in Y$, since $G$ is effective.
\end{itemize}
Therefore, the assumptions of Lemma~\ref{lem-CP-nefcase} are satisfied, and we conclude that
$$
K_{\widetilde{X}/Y} - \pi^{*}K_X + \Delta - \mathrm{Ram}(\widetilde{f}) 
\underalign{\eqref{eq-foliation}}{\sim}
K_{\widetilde{\mathcal{F}}} - \pi^{*}K_X + (\Delta - E)
$$
is psef. Since $\Delta - E$ is $\pi$-exceptional, it follows that
$$
(K_{\widetilde{\mathcal{F}}} - \pi^{*}K_X) \cdot \pi^{*}\gamma
\quad \underalign{\text{\cite[Lemma 2.2]{Ou25b}}}{=} \quad 
\left( K_{\widetilde{\mathcal{F}}} - \pi^{*}K_X + \Delta - E \right) \cdot \pi^{*}\gamma\underalign{\text{($\alpha_i$ nef)}}{\ge} 
0,
$$
which contradicts \eqref{eq-tangent-nef-1}.
\end{proof}

\begin{lem}[{cf.~\cite[Theorem 1.2]{CP25}}]
\label{lem-CP-nefcase}
Let $f: X \to Y$ be a holomorphic surjective map between compact K\"ahler manifolds. Let $M$ be a nef divisor and $L$ be a psef divisor on $X$. Assume the following:
\begin{enumerate}[label=$(\arabic*)$]
    \item The line bundle $\mathcal{O}_{X}(L)$ admits a singular Hermitian metric $h_L$ with semipositive curvature current, and the multiplier ideal sheaf of $h_L$ is trivial.
    \item The restriction $(K_X + M + L)|_{X_y}$ is psef for a general $y \in Y$.
    \item There exist a K\"ahler form $\omega$ and a holomorphic $2$-form $\sigma$ on $X$ such that $\omega + \sigma + \overline{\sigma}$ represents a rational cohomology class in $H^2(X, \Q)$ and $\sigma|_{X_y} = 0$ for a general $y \in Y$.
\end{enumerate}
Then $K_{X/Y} + M + L - \mathrm{Ram}(f)$ is psef.
Here, the ramification divisor $\mathrm{Ram}(f)$ is defined by
$$
\mathrm{Ram}(f) := \sum_{W \subset X}(m_W - 1)W,
$$
where the sum is taken over irreducible hypersurfaces $W$ in $X$ such that $Z = f(W)$ is a hypersurface in $Y$, and $m_W$ is the multiplicity of $f^{-1}(Z)$ along $W$.
\end{lem}

Note that, according to the argument in \cite[Theorem 5.1]{CP25}, if the general fiber $F$ of $f$ is rationally connected as in the above lemma, then assumption (3) is satisfied.

\begin{proof}
The proof of \cite[Theorem 3.1]{CP25} applies here with only minor modifications. We outline the necessary steps.

By \cite[Theorem 3.1, Step 1]{CP25}, there exists a finite open cover $\{ U_{\lambda} \}$ of $Y$ by coordinate charts, and for each $U_{\lambda}$, there exists an ample line bundle $\mathcal{H}_{U_{\lambda}}$ on $f^{-1}(U_{\lambda})$, together with a smooth positive metric $h_{U_{\lambda}}$ whose local weight is a potential of $\omega|_{f^{-1}(U_{\lambda})}$. Let $X_{\lambda} := f^{-1}(U_{\lambda})$. Fix a positive integer $m_0$. Since $M$ is nef, there exists a smooth metric $h_{M, m_0}$ on the line bundle $\mathcal{O}_{X}(M)$ such that
$$
\sqrt{-1} \Theta_{h_{M, m_0}} \ge - \frac{1}{m_0} \omega.
$$
Then, by \cite[Lemma 3.3 and Theorem 3.4]{CP25}, there exists a singular Hermitian metric $h_{U_{\lambda},m_0}$ with semipositive curvature current on the $\mathbb{Q}$-line bundle
$$
\mathcal{O}_{X}(K_{X_{\lambda}/U_{\lambda}} + M + L) \otimes  \mathcal{H}_{U_{\lambda}}^{\otimes \frac{1}{m_0}}.
$$
According to \cite[Theorem 3.1, Step 2]{CP25}, assumption (3) implies that the metrics $h_{U_{\lambda},m_0}$ and $h_{U_{\mu}, m_0}$ agree on overlaps. As a result, the class
$$
c_1(K_{X/Y} + M + L) + \frac{1}{m_0} \{\omega\}
$$
admits a closed positive current $\Xi_0$ on a Zariski open subset of $X$. Then, following the argument in \cite[Theorem 3.1, Steps 3 and 4]{CP25}, this current extends to a positive current $\Xi$ on $X$ satisfying $\Xi \ge [\mathrm{Ram}(f)]$. This implies that
$$
c_1(K_{X/Y} + M + L - \mathrm{Ram}(f)) + \frac{1}{m_0} \{\omega\}
$$
is psef. Letting $m_0 \to \infty$, we obtain the desired conclusion.
\end{proof}
\begin{rem}
During the preparation of an earlier version of this paper, which appeared in a previous arXiv version of \cite{IJZ25} (see also Remark~\ref{rem-earlier}), a result similar to Proposition~\ref{prop-genericnef--antinef} was independently established in \cite[Theorem~5.2]{MWWZ25} under the assumption that $X$ is a compact K\"ahler manifold. Subsequently, after \cite{IJZ25} was posted on arXiv, Proposition~\ref{prop-genericnef--antinef} was extended in \cite[Theorem~3.5]{FGSW26} to the case where $(X, D)$ is a log canonical pair.

The results in Subsection~\ref{subsec-generic-neffness-thm} remain valid if the assumption on $\gamma$ in Setup~\ref{setup-genericnef} is replaced by
$
\gamma = \langle \alpha^{n-1} \rangle,
$
where $\alpha \in H^{1,1}_{\mathrm{BC}}(X)$ is a big class satisfying the vanishing property.
(See also the previous arXiv version of \cite{IJZ25}.)
\end{rem}

\subsection{Bogomolov--Gieseker inequality for mixed polarizations}
Next, we prove the Bogomolov--Gieseker inequality for reflexive sheaves with respect to mixed polarizations.
We need the following orbifold result.

\begin{lem}[{cf.~\cite[Corollary 3.7]{CW24}}]
\label{lem-orbifold-version-CW24}
Let $X_{\mathrm{orb}}$ be a compact K\"ahler orbifold and $\omega_{1, \orb}, \ldots, \omega_{n-1,\orb}$ be orbifold K\"ahler forms. Let $(E_{\mathrm{orb}}, \theta_{\mathrm{orb}})$ be a rank $r$ Higgs orbi-bundle on $X_{\mathrm{orb}}$. If $(E_{\mathrm{orb}}, \theta_{\mathrm{orb}})$ is $\{\omega_{1,\orb}\} \cdots \{\omega_{n-1,\orb}\}$-stable, then the following Bogomolov--Gieseker inequality holds:
\[
\left( 2r\, c_2^{\mathrm{orb}}(E_{\mathrm{orb}}) - (r-1)\, c_1^{\mathrm{orb}}(E_{\mathrm{orb}})^2 \right) \cdot \{\omega_{2,\orb}\} \cdots \{\omega_{n-1,\orb}\} \ge 0.
\]
\end{lem}
For the definition of Higgs sheaves, we refer to \cite[Subsubsection 2.2.3]{ZZZ25} in the orbifold setting and to \cite[Subsubsection 3.2.2]{IJZ25} in the normal analytic setting. Since the definition of Higgs bundles is not essential in this paper, we omit it.

\begin{proof}
Since $(E_{\mathrm{orb}}, \theta_{\mathrm{orb}})$ is $\{\omega_{1,\orb}\} \cdots \{\omega_{n-1,\orb}\}$-stable, there exists an orbifold Hermitian metric $h_{\mathrm{orb}}$ satisfying the Hermite--Einstein equation:
\begin{equation}
\label{eq-CW24type-HE}
\sqrt{-1} F_{h_\mathrm{orb}, \theta_{\orb}} \wedge \omega_{1,\orb} \wedge \cdots \wedge \omega_{n-1,\orb} = \lambda \, \mathrm{Id}_{E_{\orb}} \cdot \omega_{1,\orb}^2 \wedge \omega_{2,\orb} \wedge \cdots \wedge \omega_{n-1,\orb},
\end{equation}
for some constant $\lambda \in \mathbb{R}$. Here, $F_{h_\mathrm{orb}, \theta_{\orb}}$ denotes the curvature of the Hitchin--Simpson connection associated with $(h_\mathrm{orb}, \theta_{\orb})$ (see also \cite[Definition 3.2]{ZZZ25}).
Indeed, as in \cite[Corollary 2.4]{CW24}, one can choose a positive orbifold $(1,1)$-form $\omega_{\mathrm{orb}}$ such that
\[
\frac{\omega_{\mathrm{orb}}^{n-1}}{(n-1)!} = \omega_{1,\orb} \wedge \cdots \wedge \omega_{n-1,\orb}.
\]
Note that $\omega_{\mathrm{orb}}$ is a balanced metric, and in particular, a Gauduchon metric. According to \cite[Theorem 3.3]{ZZZ25}, a Hermite--Einstein metric exists for Higgs orbi-bundles with respect to such metrics. Applying the same argument as in \cite[Theorem 3.2]{CW24} to the orbifold setting, we obtain a metric satisfying \eqref{eq-CW24type-HE}.

Once the existence of a Hermite--Einstein metric satisfying \eqref{eq-CW24type-HE} is established, the Bogomolov--Gieseker inequality follows by repeating the argument of \cite[Corollary 3.4]{CW24} in the orbifold context.
\end{proof}

\begin{cor}
\label{cor-BGinequality-mixed}
Let $X$ be a compact normal analytic variety in Fujiki's class with quotient singularities in codimension~$2$, and let $\mathcal{E}$ be a reflexive sheaf on $X$. Let $\beta \in H^{1,1}_{BC}(X)$ be a nef class, and let $\alpha_1, \ldots, \alpha_{n-2} \in H^{1,1}_{BC}(X)$ be nef and big classes such that $\beta \cdot \alpha_1 \cdots \alpha_{n-2} \not\equiv 0$. Assume that one of the following conditions holds:
\begin{enumerate}[label=$(\arabic*)$]
\item $\mathcal{E}$ is $\beta \cdot\alpha_1 \cdots \alpha_{n-2}$-semistable.
\item $X$ has rational singularities, and there exists a Higgs field $\theta$ such that the Higgs sheaf $(\mathcal{E}, \theta)$ is $\beta \cdot\alpha_1 \cdots \alpha_{n-2}$-semistable.
\end{enumerate}
Then the Bogomolov--Gieseker inequality holds:
$$
\widehat{\Delta}(\mathcal{E}) \cdot \alpha_1 \cdots \alpha_{n-2} \ge 0.
$$
\end{cor}
In (2), the assumption that $X$ has rational singularities is needed in order to define the pullback of Higgs sheaves (see \cite[Definition 3.28]{IJZ25}).

\begin{proof}
We treat only case~(1), as case~(2) is similar. As in Subsection \ref{subsec-intersection-orbifold-nonpluripolar}, we may take a bimeromorphic morphism
\( q : Z \to X \) from a compact K\"ahler variety \( Z \) with only quotient singularities, such that \( Z \) admits a compact K\"ahler orbifold structure \( Z_{\mathrm{orb}} \), and there exists an orbi-bundle \( E_{\mathrm{orb}} \) that is semistable with respect to
$q^*\beta \cdot q^*\alpha_1 \cdots q^*\alpha_{n-2}$
and satisfies
\[
\widehat{\Delta}(\mathcal{E}) \cdot \alpha_1 \cdots \alpha_{n-2} = \widehat{\Delta}_{\mathrm{orb}}(E_{\mathrm{orb}}) \cdot q^*\alpha_1 \cdots q^*\alpha_{n-2}.
\]
Here, the orbifold Bogomolov discriminant $\widehat{\Delta}_{\mathrm{orb}}(E_{\mathrm{orb}})$ is defined analogously to Definition \ref{defn-Bogomolov-discriminant}.

As in \cite[Proposition~5.1]{Chen22} (cf.~\cite[Chapter~1, Subsection~6.c]{Nak04}, \cite[Equation (3.6.1)]{Lan04}), we consider the Jordan--Hölder filtration of \( E_{\mathrm{orb}} \) with respect to \( q^*\beta \cdot q^*\alpha_1 \cdots q^*\alpha_{n-2}\), and apply the Hodge index theorem (see Lemma~\ref{lem-orbifold-Hodge-index}) to reduce to the stable case (see also Lemma~\ref{lem-Langer-ineq} or the proof of \cite[Equation (4.16)]{IJZ25}).

We may thus assume that \( E_{\mathrm{orb}} \) is stable with respect to \( q^*\beta \cdot q^*\alpha_1 \cdots q^*\alpha_{n-2} \). Let \( \omega_{\mathrm{orb}} \) be an orbifold K\"ahler form. By \cite[Corollary~6.10]{Toma19} (see also \cite[Theorem~5]{Wu21} and \cite[Lemma~3.16]{DO23}), the bundle \( E_{\mathrm{orb}} \) is stable with respect to the product
\[
\big(q^*\beta + \varepsilon\{\omega_{\mathrm{orb}}\}\big) \cdot \big(q^*\alpha_1 + \varepsilon\{\omega_{\mathrm{orb}}\}\big) \cdots \big(q^*\alpha_{n-2} + \varepsilon\{\omega_{\mathrm{orb}}\}\big)
\]
for sufficiently small \( \varepsilon > 0 \). Thus, by Lemma~\ref{lem-orbifold-version-CW24}, we obtain
\[
\widehat{\Delta}_{\mathrm{orb}}(E_{\mathrm{orb}}) \cdot \big(q^*\alpha_1 + \varepsilon\{\omega_{\mathrm{orb}}\}\big) \cdots \big(q^*\alpha_{n-2} + \varepsilon\{\omega_{\mathrm{orb}}\}\big) \ge 0.
\]
Letting \( \varepsilon \to 0 \), we obtain the desired inequality.
\end{proof}

\subsection{Langer's inequality and semipositivity of the second Chern class}
\label{subsec-Langer}

Throughout Subsection \ref{subsec-Langer}, we consider the following setup.
\begin{setup}
\label{setup-Langer-ineq}
Let $X$ be a compact normal analytic variety in Fujiki's class with quotient singularities in codimension $2$. Assume that $X$ admits a nef and big class, and let $\alpha_1, \ldots, \alpha_{n-2} \in H^{1,1}_{BC}(X)$ be nef and big classes. Set $\polar := \alpha_1 \cdots \alpha_{n-2}$.
\end{setup}


\begin{lem}[{cf.~\cite[Theorem 5.1]{Lan04}}]
\label{lem-Langer-ineq}
Under Setup~\ref{setup-Langer-ineq}, let $\mathcal{E}$ be a reflexive sheaf of rank $r$. Consider a filtration of reflexive sheaves:
\[
0 =: \mathcal{E}_0 \subsetneq \mathcal{E}_1 \subsetneq \ldots \subsetneq \mathcal{E}_l := \mathcal{E},
\]
where $\mathcal{G}_i := \mathcal{E}_i / \mathcal{E}_{i-1}$ is a torsion-free sheaf of rank $r_i$. Then the following inequality holds:
\begin{align}
\begin{split}
\label{eq-Langer-Theorem:5.1}
\frac{\widehat{\Delta}(\mathcal{E})}{r} \cdot \polar
\ge \sum_{i=1}^{l} \frac{\widehat{\Delta}(\mathcal{G}_{i}^{\vee\vee})}{r_i} \cdot \polar
- \frac{1}{r} \sum_{1 \le i < j \le l} r_i r_j \left( \frac{\widehat{c}_1(\mathcal{G}_{i}^{\vee\vee})}{r_i} - \frac{\widehat{c}_1(\mathcal{G}_{j}^{\vee\vee})}{r_j} \right)^2 \cdot \polar.
\end{split}
\end{align}
\end{lem}

\begin{proof}
From Proposition~\ref{prop-exact-sequence}, a direct computation yields
\begin{align*}
\begin{split}
\frac{1}{r} \sum_{1 \le i < j \le l} r_i r_j \left( \frac{\widehat{c}_1(\mathcal{G}_{i}^{\vee\vee})}{r_i} - \frac{\widehat{c}_1(\mathcal{G}_{j}^{\vee\vee})}{r_j} \right)^2 \cdot\polar
& \underalign{\text{}}{=}\frac{1}{2r}\sum_{ i  = 1}^{l}\sum_{ j  = 1}^{l}r_i r_j 
\left( \frac{\widehat{c}_1(\mathcal{G}_{i}^{\vee\vee})}{r_i} - \frac{\widehat{c}_1(\mathcal{G}_{j}^{\vee\vee})}{r_j} \right)^2 \cdot\polar\\
& \underalign{\text{(Prop. \ref{prop-exact-sequence})}}{=} \left( \sum_{ i  = 1}^{l} \frac{\widehat{c}_1(\mathcal{G}_{i}^{\vee\vee})^2}{r_i} -   \frac{\widehat{c}_1(\mathcal{E})^2}{r} \right)\cdot \polar
\end{split}
\end{align*}
On the other hand, repeatedly applying Proposition~\ref{prop-exact-sequence}, we obtain
\begin{align}
\begin{split}
\label{eq-Miyaoka-2}
2 \widehat{c}_2(\mathcal{E})\cdot\polar
\underalign{\text{(Prop. \ref{prop-exact-sequence})}}{\ge}
\left( \sum_{ i  = 1}^{l } \left(2\widehat{c}_2(\mathcal{G}_i^{\vee\vee}) - \widehat{c}_1(\mathcal{G}_i^{\vee\vee})^2 \right)\right) \cdot\polar
 + \widehat{c}_1(\mathcal{E})^2\cdot\polar,
\end{split}
\end{align}
Combining these computations, inequality \eqref{eq-Langer-Theorem:5.1} follows.
\end{proof}

\begin{prop}[{cf.~ \cite[Theorem 5.1]{Lan04}}]
\label{prop-BGinequality-bigclass}
Under Setup~\ref{setup-Langer-ineq}, let $\beta$ be a nef class such that $\beta \cdot \polar \not\equiv 0$. Then for any reflexive sheaf $\mathcal{E}$ of rank $r$, the following statements hold:
\begin{enumerate}[label=$(\arabic*)$]
\item If $\beta^2 \cdot \polar > 0$, then
$$
\frac{\widehat{\Delta}(\mathcal{E})}{r} \cdot \polar
\ge
-\frac{r}{\beta^{2} \cdot \polar}
\left(\mu_{ \beta \cdot \polar }^{\max}(\mathcal{E}) - \mu_{\beta \cdot \polar}(\mathcal{E})\right)
\left(\mu_{\beta \cdot \polar}(\mathcal{E}) - \mu_{\beta \cdot \polar }^{\min}(\mathcal{E})\right).
$$

\item If $\beta^2 \cdot \polar > 0$ and $\mu^{\min}_{\beta \cdot \polar}(\mathcal{E}) \ge 0$, then
$$
2\widehat{c}_2(\mathcal{E}) \cdot \polar 
\ge 
\left(1 - \frac{1}{r}\right)\widehat{c}_1(\mathcal{E})^2 \cdot \polar 
- \left(\frac{1}{r_1} - \frac{1}{r}\right) \cdot \frac{(\widehat{c}_1(\mathcal{E}) \cdot \beta \cdot \polar)^2}{\beta^2 \cdot \polar},
$$
where $r_1$ denotes the rank of the maximal destabilizing subsheaf of $\mathcal{E}$ with respect to $\beta \cdot \polar$.

\item If $\widehat{c}_1(\mathcal{E}) \cdot \beta \cdot \polar = 0$ and $\mu^{\min}_{\beta \cdot \polar}(\mathcal{E}) \ge 0$, then
$$
2\widehat{c}_2(\mathcal{E}) \cdot \polar 
\ge 
\left(1 - \frac{1}{r}\right)\widehat{c}_1(\mathcal{E})^2 \cdot \polar.
$$
\end{enumerate}
\end{prop}

\begin{proof}
(1) Consider the Harder--Narasimhan filtration of $\mathcal{E}$ with respect to $\beta \cdot \polar$:
\[
0 =: \mathcal{E}_0 \subsetneq \mathcal{E}_1 \subsetneq \ldots \subsetneq \mathcal{E}_l := \mathcal{E},
\]
where each $\mathcal{G}_i := \mathcal{E}_i / \mathcal{E}_{i-1}$ is a torsion-free $\beta \cdot \polar$-semistable sheaf of rank $r_i$. Let $\mu := \mu_{\beta \cdot \polar}(\mathcal{E})$ and $\mu_i := \mu_{\beta \cdot \polar}(\mathcal{G}_i^{\vee\vee})$. By the property of the Harder--Narasimhan filtration, we have
\begin{equation}
\label{eq-HNfiltration-maxmin}
\mu_1 = \mu^{\max}_{\beta \cdot \polar}(\mathcal{E}), \quad
\mu_l = \mu^{\min}_{\beta \cdot \polar}(\mathcal{E}), \quad
\text{and} \quad \mu_1 > \cdots > \mu_l.
\end{equation}
By Corollary~\ref{cor-BGinequality-mixed} and Lemma~\ref{lem-Langer-ineq}, we have
\begin{equation}
\label{eq-BG-Langer}
\widehat{\Delta}(\mathcal{E}) \cdot \polar \ge
- \sum_{1 \le i < j \le l} r_i r_j \left( \frac{\widehat{c}_1(\mathcal{G}_i^{\vee\vee})}{r_i} - \frac{\widehat{c}_1(\mathcal{G}_j^{\vee\vee})}{r_j} \right)^2 \cdot \polar.
\end{equation}
Applying the Hodge index theorem (see Proposition~\ref{prop-Hodge-index}), we obtain
\begin{equation}
\label{eq-BG-Langer-1}
\left( \left( \frac{\widehat{c}_1(\mathcal{G}_i^{\vee\vee})}{r_i} - \frac{\widehat{c}_1(\mathcal{G}_j^{\vee\vee})}{r_j} \right)^2 \cdot \polar \right)
\cdot (\beta^2 \cdot \polar)
\le
\left( \left( \frac{\widehat{c}_1(\mathcal{G}_i^{\vee\vee})}{r_i} - \frac{\widehat{c}_1(\mathcal{G}_j^{\vee\vee})}{r_j} \right) \cdot \beta \cdot \polar \right)^2.
\end{equation}
It then follows from \cite[Lemma 1.4]{Lan04} that
\begin{align*}
\sum_{1 \le i < j \le l} r_i r_j \left( \frac{\widehat{c}_1(\mathcal{G}_{i}^{\vee\vee})}{r_i} - \frac{\widehat{c}_1(\mathcal{G}_{j}^{\vee\vee})}{r_j} \right)^2 \cdot  \polar 
&\underalign{\text{(\ref{eq-BG-Langer-1})}}{\le}
\frac{1}{\beta^2 \cdot \polar}
\sum_{1 \le i < j \le l} r_ir_j\left( \mu_i - \mu_j \right)^2 \\
&\underalign{\text{\cite{Lan04}}}{\le}
\frac{r^2}{ \beta^2\cdot  \polar }
\left(\mu_1 - \mu\right)
\left(\mu - \mu_l\right).
\end{align*}
Therefore, by substituting the above expression into \eqref{eq-BG-Langer} and using \eqref{eq-HNfiltration-maxmin}, we obtain the desired inequality in (1).

(2) From the assumption \( \mu^{\min}_{\beta \cdot \polar}(\mathcal{E}) \ge 0 \), we have
\( r \mu \ge r_1 \mu_1 \) and \( \mu_l \ge 0 \) by \eqref{eq-HNfiltration-maxmin}, and hence
\[
- \frac{r}{\beta^2 \cdot \polar}
(\mu_1 - \mu)(\mu - \mu_l)
\ge
- \frac{r^2 \mu^2}{\beta^2 \cdot \polar} \left( \frac{1}{r_1} - \frac{1}{r} \right).
\]
Since $r\mu = \widehat{c}_1(\mathcal{E}) \cdot \beta \cdot \polar$, the inequality in (2) follows from (1).

(3) From \eqref{eq-BG-Langer}, it suffices to show
\begin{equation}
\label{eq-BG-Langer-nu1}
\left( \frac{\widehat{c}_1(\mathcal{G}_i^{\vee\vee})}{r_i} - \frac{\widehat{c}_1(\mathcal{G}_j^{\vee\vee})}{r_j} \right)^2 \cdot \polar \le 0.
\end{equation}
Since $\widehat{c}_1(\mathcal{E}) \cdot \beta \cdot \polar = 0$, we have $\sum c_1(\mathcal{G}_i^{\vee\vee}) \cdot \beta \cdot \polar = 0$. Since $\mu^{\min}_{\beta \cdot \polar}(\mathcal{E}) \ge 0$, it follows that $c_1(\mathcal{G}_i^{\vee\vee}) \cdot \beta \cdot \polar = 0$ for all $i$. Then applying the Hodge index theorem to $\frac{\widehat{c}_1(\mathcal{G}_i^{\vee\vee})}{r_i} - \frac{\widehat{c}_1(\mathcal{G}_j^{\vee\vee})}{r_j}$ yields \eqref{eq-BG-Langer-nu1}.
\end{proof}

\begin{cor}
\label{cor-semipositivity-c2}
Under Setup~\ref{setup-Langer-ineq}, for any reflexive sheaf $\mathcal{E}$ such that $\det \mathcal{E}$ is a nef $\Q$-line bundle and $\mu_{c_1(\mathcal{E}) \cdot \polar}^{\min}(\mathcal{E}) \ge 0$, the following statements hold.
\begin{enumerate}[label=$(\arabic*)$]
\item If $c_1(\mathcal{E})^2 \cdot \polar > 0$, then the following inequality holds:
$$
\widehat{c}_2(\mathcal{E}) \cdot \polar \ge \frac{r_1 - 1}{2 r_1} c_1(\mathcal{E})^2 \cdot \polar
$$
where $r_1$ is the rank of the maximal destabilizing subsheaf of $\mathcal{E}$ with respect to $c_1(\mathcal{E}) \cdot \polar$.

\item If $c_1(\mathcal{E})^2 \cdot \polar = 0$, then the following inequality holds:
$$
\widehat{c}_2(\mathcal{E}) \cdot \polar \ge 0.
$$
\end{enumerate}
\end{cor}

\begin{proof}
(1) follows by applying Proposition~\ref{prop-BGinequality-bigclass}(2) with $\beta = c_1(\mathcal{E})$.

(2) If $c_1(\mathcal{E}) \cdot \polar \not\equiv 0$, the statement follows by applying Proposition~\ref{prop-BGinequality-bigclass}(3) with $\beta = c_1(\mathcal{E})$. In the case where $c_1(\mathcal{E}) \cdot \polar \equiv 0$, the sheaf $\mathcal{E}$ is $\alpha_1 \cdot \polar$-semistable by Remark \ref{rem-Hodge-index-num}, and the desired inequality then follows from Corollary~\ref{cor-BGinequality-mixed}.
\end{proof}

\subsection{Proofs of Theorems~\ref{thm-Miyaoka-ineq-nonuniruled}--\ref{thm-nef-anti-canonical} and Corollaries~\ref{cor-MY-canonical}--\ref{cor-MY-anti-canonical}}

We first consider the case where $-K_X$ is nef (Theorem~\ref{thm-nef-anti-canonical} and Corollary~\ref{cor-MY-anti-canonical}). This follows immediately from the previous results, as explained below.

\begin{proof}[Proof of Theorem~\ref{thm-nef-anti-canonical}]
Since $X$ admits a big class, $X$ belongs to Fujiki's class (see Definition~\ref{defn-Fujiki}). Thus, by applying Proposition~\ref{prop-genericnef--antinef} and Corollary~\ref{cor-semipositivity-c2} to the tangent sheaf $\mathcal{T}_X$, with $\gamma := c_1(-K_X) \cdot \alpha_1 \cdots \alpha_{n-2}$ and $\polar := \alpha_1 \cdots \alpha_{n-2}$, we obtain the desired inequality.
\end{proof}

\begin{proof}[Proof of Corollary~\ref{cor-MY-anti-canonical}]
Under assumption (1), since $X$ is a Moishezon K\"ahler variety with at most rational singularities, it follows from \cite[Theorem 6]{Nam01} that $X$ is projective. Since $-K_X$ is nef and big, the inequality follows from \cite[Theorem 1.2]{IJZ25} (cf. \cite[Proposition 1.1]{GKP22} and \cite[Theorem 9]{DGP24}).

In case~(2), note that \( c_1(-K_X)^n = 0 \) by \cite[Lemma~2.35]{DHP24}. By applying Theorem~\ref{thm-nef-anti-canonical} to the classes \( \alpha_i = c_1(-K_X) + \varepsilon \{ \omega \} \) for any \( \varepsilon > 0 \), we obtain
\[
\widehat{c}_2(X) \cdot (c_1(-K_X) + \varepsilon \{ \omega \})^{n-2} \ge 0.
\]
Taking \( \varepsilon \) sufficiently small and considering the term of order $\varepsilon^{n-2-\nu}$, we obtain the desired inequality.
\end{proof}

Next, we consider the case where $K_X$ is nef (Theorems~\ref{thm-Miyaoka-ineq-nonuniruled}, \ref{thm-Miyaoka-klt-kahler} and Corollary~\ref{cor-MY-canonical}). This can be shown as follows, following the ideas of \cite[Section 7]{Miy87} and \cite[Subsection 4.2]{IMM24}.

\begin{proof}[Proof of Theorem~\ref{thm-Miyaoka-ineq-nonuniruled}]
We only consider case (1), since case (2) is analogous.
Set $\polar := \alpha_1 \cdots \alpha_{n-2}$. By Proposition~\ref{prop-genericnef--nonuniruled}, we have $\mu^{\min}_{c_1(K_X) \cdot \polar}(\Omega_{X}^{[1]}) \ge 0$. If $c_1(K_X)^2 \cdot \polar = 0$, then the inequality follows from Corollary~\ref{cor-semipositivity-c2}~(2). Therefore, we may assume that $c_1(K_X)^2 \cdot \polar > 0$.
We take the Harder--Narasimhan filtration with respect to $c_1(K_X)\cdot \polar$:
\[
0 =: \mathcal{E}_0 \subsetneq \mathcal{E}_1 \subsetneq \ldots \subsetneq \mathcal{E}_l := \Omega_{X}^{[1]},
\]
where each $\mathcal{G}_i := \mathcal{E}_i / \mathcal{E}_{i-1}$ is a torsion-free $c_1(K_X)\cdot \polar$-semistable sheaf of rank $r_i$. Let $\mu := \mu_{c_1(K_X) \cdot \polar}(\Omega_{X}^{[1]})$ and $\mu_i := \mu_{c_1(K_X)\cdot \polar}(\mathcal{G}_i^{\vee\vee})$.
By the property of the Harder--Narasimhan filtration, we have
\begin{equation}
\label{eq-Miyaoka-0}
\mu^{\max}_{c_1(K_X) \cdot \polar}(\Omega_{X}^{[1]})=\mu_1 > \cdots > \mu_l= \mu^{\min}_{c_1(K_X)\cdot \polar}(\Omega_{X}^{[1]}) 
\underalign{\text{(Prop. \ref{prop-genericnef--nonuniruled})}}{\ge}0.
\end{equation}
From $\mu_i < \mu_1$, we obtain
\begin{equation}
\label{eq-Miyaoka-1}
\sum_{i=2}^{l} r_i \mu_i^2
\le \mu_1 \sum_{i=2}^{l} r_i \mu_i
= - r_1 \mu_1^2 + n \mu \mu_1.
\end{equation}
Consider a Higgs field $\theta$ on $\mathcal{E} := \mathcal{G}_1 \oplus \mathcal{O}_{X}$ given by
$$
\begin{array}{cccc}
\theta \colon  & \mathcal{E} := \mathcal{G}_1 \oplus \mathcal{O}_{X} &\rightarrow  
&\mathcal{E} \otimes \Omega_{X}^{[1]} = (\mathcal{G}_1 \oplus \mathcal{O}_{X} ) \otimes \Omega_{X}^{[1]} \\
	   & (a,b)&\mapsto& (0,a).
\end{array}
$$
By the same argument as in \cite[Proposition 2.8]{IMM24}, the Higgs sheaf $(\mathcal{E}, \theta)$ is $c_1(K_X)\cdot\polar$-stable.
The Bogomolov--Gieseker inequality (see Corollary~\ref{cor-BGinequality-mixed}) yields
\begin{equation}
\label{eq-Miyaoka-3}
2\widehat{c}_2(\mathcal{G}_{1}^{\vee\vee})\cdot\polar
\underalign{\text{(Cor. \ref{cor-BGinequality-mixed} (2))}}{\ge}
\frac{r_1}{r_1 + 1} \widehat{c}_1(\mathcal{G}_{1}^{\vee\vee})^2\cdot\polar
\quad
\text{and}
\quad
2\widehat{c}_2(\mathcal{G}_{i}^{\vee\vee})\cdot\polar
\underalign{\text{(Cor. \ref{cor-BGinequality-mixed} (1))}}{\ge}
\frac{r_i - 1}{r_i} \widehat{c}_1(\mathcal{G}_{i}^{\vee\vee})^2\cdot\polar
\end{equation}

We first treat the case where $r_1 \ge 2$. Then we have
\begin{align}
\begin{split}
\label{eq-Miyaoka-4}
2 \widehat{c}_2(X)\cdot\polar
&\underalign{\text{(\ref{eq-Miyaoka-2})}}{\ge}
\left( \sum_{ i  = 1}^{l }
\left( 2\widehat{c}_2(\mathcal{G}_{i}^{\vee\vee}) - \widehat{c}_1(\mathcal{G}_{i}^{\vee\vee})^2 \right)\right)\cdot\polar + c_1(K_X)^2\cdot\polar \\
&\underalign{\text{(\ref{eq-Miyaoka-3})}}{\ge}
- \frac{1}{r_1 + 1} \widehat{c}_1(\mathcal{G}_{1}^{\vee\vee})^2\cdot\polar
- \sum_{ i  = 2}^{l } \frac{1}{r_i} \widehat{c}_1(\mathcal{G}_{i}^{\vee\vee})^2\cdot\polar
+ c_1(K_X)^2\cdot\polar \\
&\underalign{\text{(Prop. \ref{prop-Hodge-index})}}{\ge}
- \frac{1}{c_1(K_X)^2\cdot\polar}
\left( \frac{r_1^2}{r_1 + 1}\mu_1^2 + \sum_{i=2}^{l} r_i \mu_i^2 \right) + c_1(K_X)^2\cdot\polar \\
&\underalign{\text{(\ref{eq-Miyaoka-1})}}{\ge}
\frac{1}{c_1(K_X)^2\cdot\polar}
\left( \frac{r_1}{r_1 + 1} \mu_1^2 - n\mu \mu_1 \right) + c_1(K_X)^2\cdot\polar.
\end{split}
\end{align}
Since $\mu_1 \le \frac{n}{r_1} \mu$, the expression $\frac{r_1}{r_1 + 1} \mu_1^2 - n\mu \mu_1$ attains its minimum at $\mu_1 = \frac{n}{r_1} \mu$. Using the identity $n \mu = c_1(K_X)^2 \cdot \polar$,
\begin{align*}
\begin{split}
2 \widehat{c}_2(X)\cdot\polar
&\underalign{\text{(\ref{eq-Miyaoka-4})}}{\ge}
\left( \frac{1}{r_1(r_1+1)} - \frac{1}{r_1} + 1 \right)c_1(K_X)^2\cdot\polar
\underalign{\text{(by $r_1 \ge 2$)}}{\ge}
\frac{2}{3}c_1(K_X)^2\cdot\polar.
\end{split}
\end{align*}
Hence, we obtain the desired inequality.

Next, we consider the case $r_1 = 1$. In this case, $c_1(\mathcal{G}_1)^2 \cdot \polar \le 0$ by \eqref{eq-Miyaoka-3}. Therefore, we obtain
\begin{align}
\begin{split}
\label{eq-Miyaoka-5}
2 \widehat{c}_2(X)\cdot\polar
&\underalign{\text{(\ref{eq-Miyaoka-4})}}{\ge}
- \frac{1}{c_1(K_X)^2\cdot\polar}
\left( \sum_{i=2}^{l} r_i \mu_i^2 \right) + c_1(K_X)^2\cdot\polar 
\\
&\underalign{\text{(\ref{eq-Miyaoka-1})}}{\ge}
\frac{1}{c_1(K_X)^2\cdot\polar}
\left( \mu_1^2 - n\mu \mu_1 \right) + c_1(K_X)^2\cdot\polar.
\end{split}
\end{align}
The function $\mu_1^2 - n\mu \mu_1$ attains its minimum at $\mu_1 = \frac{n}{2} \mu$. Using $n \mu = c_1(K_X)^2\cdot\polar$, we obtain
\begin{align*}
\begin{split}
2 \widehat{c}_2(X)\cdot\polar
&\underalign{\text{(\ref{eq-Miyaoka-5})}}{\ge}
\frac{3}{4}c_1(K_X)^2\cdot\polar
> \frac{2}{3}c_1(K_X)^2\cdot\polar.
\end{split}
\end{align*}
\end{proof}

Next, we prove Theorem~\ref{thm-Miyaoka-klt-kahler}. In this proof, when we say that a certain property holds for \textit{sufficiently small $\varepsilon > 0$}, we mean that there exists some $\varepsilon_0 > 0$ such that the property holds for all $0 < \varepsilon < \varepsilon_0$.

\begin{proof}[Proof of Theorem~\ref{thm-Miyaoka-klt-kahler}]
Set $\beta_{\varepsilon} := c_1(K_X) + \varepsilon \{\omega\}$. We first show that $\mu^{\min}_{c_1(K_X)\cdot \beta_\varepsilon^{n-2}}(\Omega_X^{[1]}) \ge 0$ holds for sufficiently small $\varepsilon > 0$. Suppose, for contradiction, that there exists a sequence $\varepsilon_k \to 0$ such that $\mu^{\min}_{c_1(K_X)\cdot \beta_{\varepsilon_k}^{n-2}}(\Omega_X^{[1]}) < 0$.
As in the argument of \cite[Proposition 2.3]{Cao13}, for sufficiently small $\varepsilon > 0$, the maximal destabilizing subsheaf with respect to $c_1(K_X) \cdot \beta_\varepsilon^{n-2}$ becomes independent of $\varepsilon$. Therefore, there exists a reflexive subsheaf $\mathcal{E} \subset \mathcal{T}_X$ such that
$\mu_{c_1(K_X)\cdot \beta_{\varepsilon_k}^{n-2}}(\mathcal{E}) > 0$ for all sufficiently large $k$.

Now consider the function
\[
F(x_1, \ldots, x_{n-1}) := c_1(\mathcal{E}) \cdot (c_1(K_X) + x_1 \{\omega\}) \cdots (c_1(K_X) + x_{n-1} \{\omega\}).
\]
This is a symmetric polynomial, which can be written as
\[
F(x_1, \ldots, x_{n-1}) = \sum_{i=0}^{n-1} a_i \sigma_i,
\]
where $a_i \in \mathbb{R}$ and $\sigma_i$ denotes the $i$-th elementary symmetric polynomial of $x_1, \ldots, x_{n-1}$. From the assumption that $F(0, \varepsilon_k, \ldots, \varepsilon_k) > 0$ for all $k$, it follows that there exists an index $i$ such that $a_0 = \cdots = a_{i-1} = 0$ and $a_i > 0$. Hence, for sufficiently small $\varepsilon > 0$, we also have $F(\varepsilon, \ldots, \varepsilon) > 0$. This implies that
\[
\mu^{\max}_{\beta_\varepsilon^{n-1}}(\mathcal{T}_X) \ge
\frac{ F(\varepsilon, \ldots, \varepsilon)}{n}> 0
\]
for sufficiently small $\varepsilon > 0$, which contradicts Theorem \ref{thm-Enoki-genericnef}.

The remainder of the proof follows the same line of argument as the proof of Theorem~\ref{thm-Miyaoka-ineq-nonuniruled}. Indeed, if $c_1(K_X)^2 \cdot \beta_\varepsilon^{n-2} = 0$, then the result follows from Corollary~\ref{cor-semipositivity-c2}.
 If $c_1(K_X)^2 \cdot \beta_\varepsilon^{n-2} \neq 0$, then, as in the proof of Theorem~\ref{thm-Miyaoka-ineq-nonuniruled}, we consider the Harder--Narasimhan filtration with respect to $c_1(K_X) \cdot \beta_{\varepsilon}^{n-2}$.
 The same argument applies, and we omit the details.
\end{proof}

\begin{proof}[Proof of Corollary~\ref{cor-MY-canonical}]
The result follows from \cite[Theorem 1.1]{GKPT19b} and Theorem~\ref{thm-Miyaoka-klt-kahler} by the same argument used in the proof of Corollary~\ref{cor-MY-anti-canonical}.
\end{proof}

Finally, in view of potential applications, we present a generalization of \cite[Theorem 7.11]{CHP16} and \cite[Proposition 3.12]{DO23}.

\begin{prop}[{cf.~ \cite[Theorem 7.11]{CHP16}, \cite[Proposition 3.12]{DO23}}]
Let $X$ be a compact normal analytic variety in Fujiki's class with quotient singularities in codimension $2$.
Let $\mathcal{E}$ be a reflexive sheaf of rank $r$, $\beta$ be a nef class, and $\alpha$ be a nef and big class.
Assume that, for sufficiently small $\varepsilon > 0$, the following conditions hold:
\begin{enumerate}[label=$(\arabic*)$]
    \item $\mu_{(\beta + \varepsilon \alpha)^{n-1}}^{\min}(\mathcal{E}) \ge 0$,
    \item $ \beta^2 \cdot (\beta + \varepsilon \alpha)^{n-2} \ge \widehat{c}_1(\mathcal{E}) \cdot \beta \cdot (\beta + \varepsilon \alpha)^{n-2}$.
\end{enumerate}
Then the following inequality holds:
\[
2\widehat{c}_2(\mathcal{E}) \cdot \beta^{n-2}
\ge 
\left(\widehat{c}_1(\mathcal{E})^2 - \widehat{c}_1(\mathcal{E}) \cdot \beta\right) \cdot \beta^{n-2}.
\]
In particular, if $\widehat{c}_1(\mathcal{E})^2 \cdot \beta^{n-2} \ge 0$ and $\beta^{n}=0$, then $\widehat{c}_2(\mathcal{E}) \cdot \beta^{n-2} \ge 0$.
\end{prop}

\begin{proof}
Set $\beta_{\varepsilon} := \beta + \varepsilon \alpha$. As in the proof of Theorem~\ref{thm-Miyaoka-klt-kahler}, we have
$\mu^{\min}_{\beta \cdot \beta_{\varepsilon}^{n-2}}(\mathcal{E}) \ge 0$ for sufficiently small $\varepsilon > 0$.
Assume first that $\beta^2 \cdot \beta_{\varepsilon}^{n-2} > 0$ for sufficiently small $\varepsilon > 0$. Then, by Proposition~\ref{prop-BGinequality-bigclass} (2), we obtain
\begin{equation}
\label{eq-CHPDO-1}
(2\widehat{c}_2(\mathcal{E}) - \widehat{c}_1(\mathcal{E})^2) \cdot \beta_{\varepsilon}^{n-2}
\ge
\frac{1}{r} \left(\frac{(\widehat{c}_1(\mathcal{E}) \cdot \beta \cdot \beta_{\varepsilon}^{n-2})^2}{\beta^2 \cdot \beta_{\varepsilon}^{n-2}} -\widehat{c}_1(\mathcal{E})^2 \cdot \beta_{\varepsilon}^{n-2}\right)
- 
\frac{(\widehat{c}_1(\mathcal{E}) \cdot \beta \cdot \beta_{\varepsilon}^{n-2})^2}{\beta^2 \cdot \beta_{\varepsilon}^{n-2}}.
\end{equation}
(Note that in Proposition~\ref{prop-BGinequality-bigclass} (2), we have $r_1 \ge 1$.) By applying the Hodge index theorem (see Proposition~\ref{prop-Hodge-index}) and using assumption (2), we obtain
\[
\frac{(\widehat{c}_1(\mathcal{E}) \cdot \beta \cdot \beta_{\varepsilon}^{n-2})^2}{\beta^2 \cdot \beta_{\varepsilon}^{n-2}} 
\underset{\text{(Prop. \ref{prop-Hodge-index})}}{\ge}
\widehat{c}_1(\mathcal{E})^2 \cdot \beta_{\varepsilon}^{n-2} 
\quad
\text{and}
\quad
\widehat{c}_1(\mathcal{E}) \cdot \beta \cdot \beta_{\varepsilon}^{n-2}
\underset{(2)}{\ge}
\frac{(\widehat{c}_1(\mathcal{E}) \cdot \beta \cdot \beta_{\varepsilon}^{n-2})^2}{\beta^2 \cdot \beta_{\varepsilon}^{n-2}},
\]
which together yield the desired inequality from \eqref{eq-CHPDO-1} by taking the limit as $\varepsilon \to 0$.

Now assume instead that $\beta^2 \cdot \beta_{\varepsilon_k}^{n-2} = 0$ for some sequence $\{\varepsilon_k\}_{k=1}^{\infty}$ of positive numbers with $\varepsilon_k \to 0$ as $k \to \infty$. Then, from assumption (2), we obtain
$0 \ge \widehat{c}_1(\mathcal{E}) \cdot \beta \cdot \beta_{\varepsilon_k}^{n-2}$, and thus assumption (1) implies
\begin{equation}
\label{eq-CHPDO-lemma}
\widehat{c}_1(\mathcal{E}) \cdot \beta \cdot \beta_{\varepsilon_k}^{n-2} = 0.
\end{equation}
By the Hodge index theorem, it follows that
$\widehat{c}_1(\mathcal{E})^2 \cdot \beta_{\varepsilon_k}^{n-2} \le 0$. Therefore, by applying Proposition~\ref{prop-BGinequality-bigclass} (3), we get
\[
(2\widehat{c}_2(\mathcal{E}) - \widehat{c}_1(\mathcal{E})^2) \cdot \beta_{\varepsilon_k}^{n-2}
\underalign{\text{(Prop. \ref{prop-BGinequality-bigclass})}}{\ge} 
 - \frac{1}{r}  \widehat{c}_1(\mathcal{E})^2 \cdot \beta_{\varepsilon_k}^{n-2}
\ge 0.
\]
The desired inequality then follows from \eqref{eq-CHPDO-lemma} by letting $k \to \infty$.
\end{proof}

\section{Further remarks}
\label{sec-further-discussion}
\subsection{The case where the canonical divisor is psef}

We are interested in whether Miyaoka's inequality (cf.~Theorems~\ref{thm-Miyaoka-ineq-nonuniruled}--\ref{thm-nef-anti-canonical}) can be extended to the case where $K_X$ or $-K_X$ is psef. However, even if $-K_X$ is psef, the generic nefness theorem (cf.~Proposition~\ref{prop-genericnef--antinef}) does not necessarily hold. Indeed, as noted in \cite[Section 1]{Ou17}, let $C$ be a curve of genus $g \ge 2$, and let $L$ be a line bundle on $C$ of degree less than $2 - 3g$. Consider the projective bundle
\(
X := \mathbb{P}(\mathcal{O}_C \oplus L).
\)
In this case, $-K_X$ is psef, but $\mathcal{T}_X$ admits the pull-back of $\mathcal{T}_C$ as a quotient sheaf. Thus, $\mathcal{T}_X$ is not generically nef.

On the other hand, if $K_X$ is psef, then the generic nefness theorem (cf.~Proposition~\ref{prop-genericnef--nonuniruled}) does hold, allowing us to pursue results in this setting. In fact, in the surface case, Langer \cite{Lan03} proved the following inequality:
\begin{thm}\cite[Theorem 0.1]{Lan03}
Let $(X, D)$ be a normal log canonical pair with $X$ a projective surface. If $K_X + D$ is psef, then
\[
3c_2(X, D) \ge c_1(K_X + D)^2.
\]
Moreover, if equality holds, then $K_X + D$ is nef.
\end{thm}
In higher dimensions, Rousseau and Taji \cite[Theorem 1.2 and Subsection 4.3]{RT22} considered a log smooth pair $(X, D)$ with $K_X + D$ psef. They proved that there exists a decomposition $K_X + D= P + N$ such that $P \cdot N \cdot H^{n-2} = 0$ for some ample divisor $H$, and established the following version of Miyaoka's inequality:
\[
\left(3\widehat{c}_2(X,D) - \widehat{c}_1(X,D)^2\right) \cdot H^{n-2}
\ge \frac{1}{2} N^2 \cdot H^{n-2}.
\]
However, the decomposition $K_X + D= P + N$ in this context is not necessarily the divisorial Zariski decomposition of $K_X$. This leads to the following natural question:

\begin{ques}
Can the results of \cite{Lan03} and \cite{RT22} be extended to higher dimensions using the divisorial Zariski decomposition \( K_X = P(K_X) + N(K_X) \)?
\end{ques}

In fact, we can obtain a similar inequality in the following setting.
Let $X$ be a non-uniruled klt K\"ahler variety with $K_X$ psef. If the positive part $P(K_X)$ of $K_X$ is nef and satisfies $P(K_X)^{n-1} \neq 0$, then
\[
\left(3\widehat{c}_2(X) - \widehat{c}_1(X)^2\right) \cdot P(K_X)^{n-2}
\ge \frac{1}{2} N(K_X)^2 \cdot P(K_X)^{n-2}
\]
holds. This follows by considering the Harder--Narasimhan filtration with respect to $P(K_X)^{n-1}$ instead of $c_1(K_X) \cdot \alpha_1 \cdots \alpha_{n-2}$ in Theorem~\ref{thm-Miyaoka-ineq-nonuniruled}, and applying the same argument. However, this result remains unsatisfactory.
In addition, it remains unknown whether Proposition~\ref{prop-exact-sequence} holds for the non-pluripolar product $\langle \alpha_1 \cdots \alpha_{n-2} \rangle$, even in the special case $\langle \alpha^{n-2} \rangle$. 
Therefore, in order to extend the results of Section~\ref{sec-Miyaoka's-inequality} to the psef case, this issue also needs to be resolved.

\subsection{On the equality cases}

\begin{ques}
What is the structure of $X$ when equality holds in Theorems~\ref{thm-Miyaoka-ineq-nonuniruled}--\ref{thm-nef-anti-canonical}?
\end{ques}
When $X$ is a projective klt variety, the structure is already well understood (see \cite{Cao13, Ou17, IM22, IMM24}). If $X$ is a compact K\"ahler manifold, one can apply the arguments from \cite{IM22} and \cite{IMM24} directly to obtain the corresponding structure theorem.

\begin{thm}
\label{thm-main-structure}
Let $X$ be a compact K\"ahler manifold and let $\omega_1, \ldots, \omega_{n-2}$ be K\"ahler forms. If $K_X$ is nef and
$$
\left( 3c_2(X) - c_1(X)^2 \right) \cdot \{\omega_1\} \cdots \{\omega_{n-2}\} = 0,
$$
then the canonical divisor $K_X$ is semi-ample, and $\nu(K_X) = \kappa(K_X)$ is either $0$, $1$, or $2$. Moreover, up to a finite \'{e}tale cover of $X$, one of the following holds depending on the Kodaira dimension:
\begin{itemize}
\item[$(i)$] If $\nu(K_X) = \kappa(K_X) = 0$, then $X$ is isomorphic to a complex torus.

\item[$(ii)$] If $\nu(K_X) = \kappa(K_X) = 1$, then $X$ admits a smooth torus fibration $X \rightarrow C$ over a curve $C$ of genus $\ge 2$.

\item[$(iii)$] If $\nu(K_X) = \kappa(K_X) = 2$, then $X$ is isomorphic to the product $A \times S$, where $A$ is a complex torus and $S$ is a smooth projective surface whose universal cover is an open ball in $\mathbb{C}^2$.
\end{itemize}
\end{thm}
The reason why the arguments in \cite{IM22} and \cite{IMM24} can be extended to the compact K\"ahler setting is that the key tools used in those proofs are also available in the compact K\"ahler case. In particular, the generic nefness theorem has already been established by \cite{Ou25} and \cite{CP25}, and essential ingredients such as the Hodge index theorem for $\{\omega_1\} \cdots \{\omega_{n-2}\}$ have been proven in this paper. Furthermore, the structure of reflexive sheaves that attain equality in the Bogomolov--Gieseker inequality, namely, numerically projectively flat bundles, is known thanks to \cite{CW24}. 

Similarly, it was shown in \cite[Theorem 5.3]{MWWZ25} that a structure theorem analogous to those in \cite{Cao13, Ou17, IM22} holds when \( X \) is a compact K\"ahler manifold and \( -K_X \) is nef.

However, the case where $X$ is a klt K\"ahler variety has not yet been resolved. There seem to be the following obstacles, judging from the proof of \cite{IMM24}. When $-K_X$ is nef, a structure theorem in the spirit of \cite{CH19} and \cite{MW21} is not yet available. When $K_X$ is nef, the corresponding generic nefness theorem for mixed polarizations has not yet been established.

\bibliographystyle{alpha}
\bibliography{ref_Miyaoka_2026.bib}

\begin{thebibliography}{MWWZ25}

\bibitem[BG13]{BG13}
S\'ebastien Boucksom and Vincent Guedj.
\newblock Regularizing properties of the {K}\"ahler-{R}icci flow.
\newblock In {\em An introduction to the {K}\"ahler-{R}icci flow}, volume 2086
  of {\em Lecture Notes in Math.}, pages 189--237. Springer, Cham, 2013.

\bibitem[Cam11]{Cam11}
Fr\'{e}d\'{e}ric Campana.
\newblock Remarks on an example of {K}. {U}eno.
\newblock In {\em Classification of algebraic varieties}, EMS Ser. Congr. Rep.,
  pages 115--121. Eur. Math. Soc., Z\"{u}rich, 2011.

\bibitem[Cao13]{Cao13}
Junyan Cao.
\newblock A remark on compact {K}\"ahler manifolds with nef anticanonical
  bundles and its applications, 2013.
\newblock Preprint. arXiv:1305.4397.

\bibitem[CH19]{CH19}
Junyan Cao and Andreas H\"{o}ring.
\newblock A decomposition theorem for projective manifolds with nef
  anticanonical bundle.
\newblock {\em J. Algebraic Geom.}, 28(3):567--597, 2019.

\bibitem[Che25]{Chen22}
Xuemiao Chen.
\newblock Admissible {Hermitian}-{Yang}-{Mills} connections over normal
  varieties.
\newblock {\em Math. Ann.}, 392(1):487--523, 2025.

\bibitem[CHP16]{CHP16}
Fr\'{e}d\'{e}ric Campana, Andreas H\"{o}ring, and Thomas Peternell.
\newblock Abundance for {K}\"{a}hler threefolds.
\newblock {\em Ann. Sci. \'{E}c. Norm. Sup\'{e}r. (4)}, 49(4):971--1025, 2016.

\bibitem[Cla17]{Cla16}
Beno{\^{\i}}t Claudon.
\newblock Positivity of the logarithmic cotangent and {Shafarevich}-{Viehweg}
  conjecture.
\newblock In {\em S\'eminaire Bourbaki. Volume 2015/2016. Expos\'es 1104--1119.
  Avec table par noms d'auteurs de 1948/49 \`a 2015/16}, pages 27--63, ex.
  Paris: Soci{\'e}t{\'e} Math{\'e}matique de France (SMF), 2017.

\bibitem[CP19]{CP19}
Fr\'{e}d\'{e}ric Campana and Mihai P\u{a}un.
\newblock Foliations with positive slopes and birational stability of orbifold
  cotangent bundles.
\newblock {\em Publ. Math. Inst. Hautes \'{E}tudes Sci.}, 129:1--49, 2019.

\bibitem[CP25]{CP25}
Junyan Cao and Mihai P\u{a}un.
\newblock Remarks on relative canonical bundles and algebraicity criteria for
  foliations in {K}ähler context, 2025.
\newblock Preprint. arXiv:2502.02183.

\bibitem[CW24]{CW24}
Xuemiao Chen and Richard~A. Wentworth.
\newblock The nonabelian {Hodge} correspondence for balanced {Hermitian}
  metrics of {Hodge}-{Riemann} type.
\newblock {\em Math. Res. Lett.}, 31(3):639--654, 2024.

\bibitem[DGP24]{DGP24}
St{\'e}phane Druel, Henri Guenancia, and Mihai P{\u{a}}un.
\newblock A decomposition theorem for {{\(\mathbb{Q}\)}}-{Fano}
  {K{\"a}hler}-{Einstein} varieties.
\newblock {\em C. R., Math., Acad. Sci. Paris}, 362(S1):93--118, 2024.

\bibitem[DH23]{DH23}
Omprokash Das and Christopher Hacon.
\newblock On the minimal model program for {K}\"ahler 3-folds, 2023.
\newblock Preprint. arXiv:2306.11708.

\bibitem[DH25]{DH20}
Omprokash Das and Christopher Hacon.
\newblock The log minimal model program for {K{\"a}hler} 3-folds.
\newblock {\em J. Differ. Geom.}, 130(1):151--207, 2025.

\bibitem[DHP24]{DHP24}
Omprokash Das, Christopher Hacon, and Mihai P{\u{a}}un.
\newblock On the 4-dimensional minimal model program for {K{\"a}hler}
  varieties.
\newblock {\em Adv. Math.}, 443:68, 2024.
\newblock Id/No 109615.

\bibitem[DO23]{DO23}
Omprokash Das and Wenhao Ou.
\newblock On the log abundance for compact {K}\"ahler threefolds ii, 2023.
\newblock Preprint. arXiv:2306.00671. To appear in Proceedings of the London
  Mathematical Society.

\bibitem[Eno88]{Eno88}
Ichiro Enoki.
\newblock Stability and negativity for tangent sheaves of minimal {K}\"ahler
  spaces.
\newblock In {\em Geometry and analysis on manifolds ({K}atata/{K}yoto, 1987)},
  volume 1339 of {\em Lecture Notes in Math.}, pages 118--126. Springer,
  Berlin, 1988.

\bibitem[FGSW26]{FGSW26}
Xin Fu, Bin Guo, Jian Song, and Juanyong Wang.
\newblock Fundamental groups of compact {K}\"{a}hler varieties with nef anti
  canonical bundle, 2026.
\newblock Preprint. arXiv:2602.07420.

\bibitem[GK20]{GK20}
Patrick Graf and Tim Kirschner.
\newblock Finite quotients of three-dimensional complex tori.
\newblock {\em Ann. Inst. Fourier (Grenoble)}, 70(2):881--914, 2020.

\bibitem[GKP16]{GKP16a}
Daniel Greb, Stefan Kebekus, and Thomas Peternell.
\newblock Movable curves and semistable sheaves.
\newblock {\em Int. Math. Res. Not. IMRN}, 2016(2):536--570, 2016.

\bibitem[GKP22]{GKP22}
Daniel Greb, Stefan Kebekus, and Thomas Peternell.
\newblock Projective flatness over klt spaces and uniformisation of varieties
  with nef anti-canonical divisor.
\newblock {\em J. Algebraic Geom.}, 31(3):467--496, 2022.

\bibitem[GKPT19]{GKPT19b}
Daniel Greb, Stefan Kebekus, Thomas Peternell, and Behrouz Taji.
\newblock The {M}iyaoka-{Y}au inequality and uniformisation of canonical
  models.
\newblock {\em Ann. Sci. \'{E}c. Norm. Sup\'{e}r. (4)}, 52(6):1487--1535, 2019.

\bibitem[GP24]{GP24}
Henri Guenancia and Mihai P\u{a}un.
\newblock {B}ogomolov-{G}ieseker inequality for log terminal {K}\"ahler
  threefolds, 2024.
\newblock Preprint. arXiv:2405.10003.

\bibitem[Gue16]{Gue16}
Henri Guenancia.
\newblock Semistability of the tangent sheaf of singular varieties.
\newblock {\em Algebr. Geom.}, 3(5):508--542, 2016.

\bibitem[His24]{Hisa24}
Tomoyuki Hisamoto.
\newblock On the {M}iyaoka-{Y}au inequality for manifolds with nef
  anti-canonical line bundle, 2024.
\newblock Preprint. arXiv:2403.09120.

\bibitem[HP16]{HP16}
Andreas H\"oring and Thomas Peternell.
\newblock Minimal models for {K}\"ahler threefolds.
\newblock {\em Invent. Math.}, 203(1):217--264, 2016.

\bibitem[IJZ25]{IJZ25}
Masataka Iwai, Satoshi Jinnouchi, and Shiyu Zhang.
\newblock The miyaoka-yau inequality for singular varieties with big canonical
  or anticanonical divisors, 2025.
\newblock Preprint. arXiv:2507.08522.

\bibitem[IM22]{IM22}
Masataka Iwai and Shin-ichi Matsumura.
\newblock Abundance theorem for minimal compact {K}\"ahler manifolds with
  vanishing second {C}hern class, 2022.
\newblock Preprint. arXiv:2205.10613.

\bibitem[IMM24]{IMM24}
Masataka Iwai, Shin-ichi Matsumura, and Niklas M\"uller.
\newblock Minimal projective varieties satisfying {M}iyaoka's equality, 2024.
\newblock Preprint. arXiv:2404.07568 To appear in Proceedings of the London
  Mathematical Society.

\bibitem[Kaw92]{Kaw92}
Yujiro Kawamata.
\newblock Abundance theorem for minimal threefolds.
\newblock {\em Invent. Math.}, 108(2):229--246, 1992.

\bibitem[KM98]{KM98}
J\'{a}nos Koll\'{a}r and Shigefumi Mori.
\newblock {\em Birational geometry of algebraic varieties}, volume 134 of {\em
  Cambridge Tracts in Mathematics}.
\newblock Cambridge University Press, Cambridge, 1998.
\newblock With the collaboration of C. H. Clemens and A. Corti, Translated from
  the 1998 Japanese original.

\bibitem[KO25]{KO25}
János Kollár and Wenhao Ou.
\newblock Orbifold modifications of complex analytic spaces, 2025.
\newblock Preprint. arXiv:2512.20708.

\bibitem[Kob14]{Kob14}
Shoshichi Kobayashi.
\newblock {\em Differential geometry of complex vector bundles}.
\newblock Princeton Legacy Library. Princeton University Press, Princeton, NJ,
  [2014].
\newblock Reprint of the 1987 edition [ MR0909698].

\bibitem[Lan03]{Lan03}
Adrian Langer.
\newblock Logarithmic orbifold {Euler} numbers of surfaces with applications.
\newblock {\em Proc. Lond. Math. Soc. (3)}, 86(2):358--396, 2003.

\bibitem[Lan04]{Lan04}
Adrian Langer.
\newblock Semistable sheaves in positive characteristic.
\newblock {\em Ann. of Math. (2)}, 159(1):251--276, 2004.

\bibitem[Ma05]{ma2005orbifolds}
Xiaonan Ma.
\newblock Orbifolds and analytic torsions.
\newblock {\em Transactions of the American Mathematical Society},
  357(6):2205--2233, 2005.

\bibitem[Miy77]{Miy77}
Yoichi Miyaoka.
\newblock On the {C}hern numbers of surfaces of general type.
\newblock {\em Invent. Math.}, 42:225--237, 1977.

\bibitem[Miy87]{Miy87}
Yoichi Miyaoka.
\newblock The {C}hern classes and {K}odaira dimension of a minimal variety.
\newblock In {\em Algebraic geometry, {S}endai, 1985}, volume~10 of {\em Adv.
  Stud. Pure Math.}, pages 449--476. North-Holland, Amsterdam, 1987.

\bibitem[Miy88a]{Miy88b}
Yoichi Miyaoka.
\newblock Abundance conjecture for {$3$}-folds: case {$\nu=1$}.
\newblock {\em Compositio Math.}, 68(2):203--220, 1988.

\bibitem[Miy88b]{Miy88a}
Yoichi Miyaoka.
\newblock On the {K}odaira dimension of minimal threefolds.
\newblock {\em Math. Ann.}, 281(2):325--332, 1988.

\bibitem[MTTW25]{ma2025superconnection}
Qiaochu Ma, Xiang Tang, Hsian-Hua Tseng, and Zhaoting Wei.
\newblock Superconnection and orbifold {C}hern character, 2025.
\newblock Preprint. arXiv:2505.13912.

\bibitem[MW21]{MW21}
Shin-ichi Matsumura and Juanyong Wang.
\newblock Structure theorem for projective klt pairs with nef anti-canonical
  divisor, 2021.
\newblock Preprint. arXiv: 2105.14308. To appear in Journal of European
  Mathematical Society.

\bibitem[MWWZ25]{MWWZ25}
Shin-ichi Matsumura, Juanyong Wang, Xiaojun Wu, and Qimin Zhang.
\newblock Compact {K}\"ahler manifolds with nef anti-canonical bundle, 2025.
\newblock Preprint. arXiv: 2506.23218.

\bibitem[Nak04]{Nak04}
Noboru Nakayama.
\newblock {\em Zariski-decomposition and abundance}, volume~14 of {\em MSJ
  Memoirs}.
\newblock Mathematical Society of Japan, Tokyo, 2004.

\bibitem[Nam02]{Nam01}
Yoshinori Namikawa.
\newblock Projectivity criterion of {Moishezon} spaces and density of
  projective symplectic varieties.
\newblock {\em Int. J. Math.}, 13(2):125--135, 2002.

\bibitem[Ou23]{Ou17}
Wenhao Ou.
\newblock On generic nefness of tangent sheaves.
\newblock {\em Math. Z.}, 304(4):58, 2023.

\bibitem[Ou24]{Ou24}
Wenhao Ou.
\newblock Orbifold modifications of complex analytic varieties, 2024.
\newblock Preprint. arXiv:2401.07273.

\bibitem[Ou25a]{Ou25}
Wenhao Ou.
\newblock A characterization of uniruled compact {K}\"ahler manifolds, 2025.
\newblock Preprint. arXiv:2501.18088.

\bibitem[Ou25b]{Ou25b}
Wenhao Ou.
\newblock Orbifold chern classes and bogomolov-gieseker inequalities, 2025.
\newblock Preprint. arXiv:2512.22273.

\bibitem[RT23]{RT22}
Erwan Rousseau and Behrouz Taji.
\newblock Chern class inequalities for nonuniruled projective varieties.
\newblock {\em Bull. Lond. Math. Soc.}, 55(6):2856--2875, 2023.

\bibitem[Sat56]{Sat56}
Ichir{\^o} Satake.
\newblock On a generalization of the notion of manifold.
\newblock {\em Proc. Natl. Acad. Sci. USA}, 42:359--363, 1956.

\bibitem[Tom21]{Toma19}
Matei Toma.
\newblock Bounded sets of sheaves on relative analytic spaces.
\newblock {\em Ann. Henri Lebesgue}, 4:1531--1563, 2021.

\bibitem[Wu21]{Wu21}
Xiaojun Wu.
\newblock The {B}ogomolov’s inequality on a singular complex space, 2021.
\newblock Preprint. arXiv:2106.14650.

\bibitem[Wu22]{Wu22}
Xiaojun Wu.
\newblock Strongly pseudo-effective and numerically flat reflexive sheaves.
\newblock {\em J. Geom. Anal.}, 32(4):Paper No. 124, 61, 2022.

\bibitem[Zha24]{Zh24}
Yashan Zhang.
\newblock A note on {T}eissier problem for nef classes.
\newblock {\em Proc. Amer. Math. Soc.}, 152(7):2831--2843, 2024.

\bibitem[ZZZ25]{ZZZ25}
Chuangjing Zhang, Shiyu Zhang, and Xi~Zhang.
\newblock The {M}iyaoka-{Y}au inequality for minimal {K}\"ahler klt spaces,
  2025.
\newblock Preprint. arXiv:2503.13365.

\end{thebibliography}
\end{document}